\documentclass[11pt,a4paper]{article}
\usepackage{lmodern}
\usepackage[english]{babel}
\usepackage[utf8]{inputenc}
\usepackage{microtype}
\usepackage{hyperref}
\usepackage{setspace}
\usepackage{amsmath}
\usepackage{amssymb}
\usepackage{amsthm}
\usepackage{mathrsfs}
\usepackage{booktabs}
\usepackage{caption}
\usepackage{subfig}
\usepackage{graphicx}
\usepackage{algorithm,algpseudocode}
\usepackage{multirow}
\usepackage{yhmath}
\usepackage{xcolor}
\usepackage{geometry}
\usepackage{setspace}
\algrenewcommand\alglinenumber[1]{\footnotesize #1}

\newtheorem{remark}{Remark}
\newtheorem{lemma}{Lemma}

\begin{document}

\title{A new stopping criterion for Krylov solvers\\ applied in Interior Point Methods}
\author{Filippo Zanetti\footnote{School of Mathematics, University of Edinburgh, Edinburgh, UK. \href{mailto:f.zanetti@sms.ed.ac.uk}{f.zanetti@sms.ed.ac.uk}} \and Jacek Gondzio\footnote{School of Mathematics, University of Edinburgh, Edinburgh, UK. \href{mailto:j.gondzio@ed.ac.uk}{j.gondzio@ed.ac.uk}}}
\date{}
\maketitle

\begin{abstract}
\noindent 
When an iterative method is applied to solve the linear equation system 
in interior point methods (IPMs), the attention is usually placed on accelerating their convergence 
by designing appropriate preconditioners, but the linear solver is applied 
as a black box with a standard termination criterion which asks 
for a sufficient reduction of the residual in the linear system. 
Such an approach often leads to an unnecessary ''oversolving'' of linear 
equations. 
In this paper, an IPM that relies on an inner termination criterion not based on the residual of the linear system is introduced and analyzed. Moreover, 
new indicators for the early termination of the inner iterations are derived 
from a deep understanding of IPM needs.
The new technique has been adapted to the Conjugate Gradient (CG) 
and to the Minimum Residual method (MINRES) applied in the IPM context. 
The new criterion has been tested on a set of quadratic optimization 
problems including compressed sensing, image processing and instances 
with partial differential equation constraints, and it has been
compared to standard residual tests with variable tolerance.
Evidence gathered from these computational experiments shows that the new 
technique delivers significant improvements in terms of inner (linear) 
iterations and those translate into significant savings of the IPM 
solution time.
\end{abstract}

\noindent\textbf{Keywords}: Quadratic Programming, Interior Point Methods, Conjugate Gradient, MINRES, Stopping criterion.

\section{Introduction}

Interior Point Methods (IPMs) represent the state-of-the-art for the solution of convex optimization problems. Being second-order methods, they usually converge in merely a few
iterations and if the cost of a single iteration is kept small they are able to outperform the first-order methods, especially when it comes to problems of very large dimensions. In these instances, the linear system that arises at each iteration is usually solved with an iterative Krylov subspace method (see e.g. \cite{kelley} or \cite{regmi2,regmi} for a newer analysis), either Conjugate Gradient or MINRES, depending on the approach chosen. The ill-conditioning of the matrices involved has given rise to a wide collection of preconditioning strategies for various applications of IPMs (e.g. \ \cite{ip_pmm,berga_indefinite,iter:BCO-COAP,diserafino_orban,mfcs,mfipm,pearson_pde}). Recent developments have also been made regarding the impact of the numerical linear algebra inexactness (e.g.\ \cite{ddd_mutual,unreduced}) and in particular on the effect that an inexact linear solver has on the convergence properties of IPMs (e.g.\ \cite{baryamureeba,inexactipm,cafieri2,cui_morikuni,freund_jarre_mizuno,gondzio_inexact,korzak,lu_monteiro_oneal,zhou_toh}).

When an iterative linear solver is used, the common approach is to employ a stopping criterion based on the reduction of the residual, i.e.\ the internal solver is stopped as soon as the initial residual is reduced by a certain predetermined factor. Different strategies have been developed in order to choose a stopping tolerance that allows the outer IPM iterations to converge, without requiring too many inner (linear solver) iterations (e.g.\ \cite{stop_inner,mizuno_jarre,inexact_morini}). However, these techniques always rely on a tolerance imposed on the residual of the linear system, although with a varying relative reduction requested. This approach does not necessarily represent the best choice, since the overall goal is not obtaining an accurate solution to the sequence of linear systems, but finding a suitable (though inexact) search direction for the optimization problem; in particular, it may be possible to obtain an inexact Newton direction that would be considered too rough from a purely linear algebra perspective (i.e.\ its residual still would be too high and any standard stopping criterion would reject it) but that could be good enough to perform the next iteration of IPM successfully (i.e.\ the direction guarantees sufficient reductions of infeasibilities and the duality gap). 
Deriving a stop criterion that accepts a direction not based on its {\it residual}, but based on a {\it potential improvement} it can bring to the outer IPM iterations could reduce the number of inner iterations required at each outer step, with little or no disadvantage to the overall convergence properties of the IPM. 

Specialized early stopping strategies have been used in other fields: in~\cite{notay_cg,stathopoulos} a CG stop criterion is applied to the Jacobi-Davidson eigensolver, when finding eigenvalues of large matrices; in \cite{freitag_spence} a termination criterion is applied to inverse iterations for solving generalized eigenvalue problems; early stopping is also used in inverse problems and machine learning as a regularizer, to avoid the phenomenon known as semiconvergence (see e.g.\ \cite{semiconvergence}); in~\cite{stop_fem,stop_axelsson,silvester_simoncini} other stopping criteria are derived for various applications. However, to the best of the authors' knowledge, an early stopping criterion, specifically designed to be applied in IPMs, that does not rely entirely on the residual of the linear system, has not yet been derived. This paper fills the gap. 

In order to obtain such a criterion, the convergence indicators of IPM (i.e. primal and dual infeasibility and complementarity) need to be estimated while performing the inner iterations with CG or MINRES. The main problem is that, to compute the complete primal-dual Newton direction and to compute the infeasibilities, additional matrix-vector products would be required at each inner iteration. Since in general only one matrix-vector product and one preconditioner application per iteration are performed, adding extra matrix applications would considerably slow down the linear solver. Fortunately, with some judicious implementation and exploiting the matrix operations that are already executed, the IPM convergence indicators can be estimated using only vector operations, resulting in a minimal increase in the cost of a single linear iteration.


From the theoretical point of view, the authors introduce an ideal stopping criterion that does not rely on the reduction of the residual of the linear system This is novel with respect to the standard literature on inexact IPMs, that mostly focuses on choosing the appropriate sequence of tolerances for each IPM iteration. A sketch of complexity analysis is given and the authors provide a rationale as to why the algorithm is expected to perform similarly to the exact version. The main assumptions used are boundedness of the iterates and the ability of the chosen Krylov solver to produce a direction that satisfies the stopping criterion; a rigorous proof of this fact is difficult, due to the complicated interaction between IPM and linear solvers, and is left as an item for further research.

The paper also introduces new indicators to estimate the optimal stopping point for the inner linear iterations and analyzes their behaviour in comparison to the residual of the linear system for the problems considered. The empirical evidence suggests a new technique to terminate early the inner iterations, which is mainly based on these new indicators rather than on a sequence of residual tolerances. Although the theoretical and practical stopping criteria are different and the complexity analysis does not directly apply to the method used in the empirical section, the proposed practical termination indicators are strongly influenced by the theoretical results.

The resulting algorithms for the solution of the linear systems are called {\it Interior Point Conjugate Gradient} (IPCG) or {\it Interior Point MINRES} (IPMINRES), depending on the approach chosen: they are specialized for the specific task which needs to be solved and show significant improvements with respect to the standard CG or MINRES on the problems that were considered, which include quadratic programs derived from image processing, compressed sensing and Partial Differential Equation (PDE) constrained optimization. In particular, the new strategy is able to avoid unnecessary inner iterations in the early stage of the IPM, while retaining the good behaviour of the method in its late iterations.

The rest of the paper is organized as follows: in Section~\ref{section_ipm} the Interior Point Method is described; in Section~\ref{section_ipcg} the new IPCG and IPMINRES iterations, that allow to estimate the convergence of IPM, are introduced; Section~\ref{section_stop} introduces a theoretical stopping criterion, for which the complexity analysis is performed, and the new indicators used in practice; in Section~\ref{section_results} the test problems and numerical results are presented.

\paragraph{Notation}
In the following, $e$ indicates the vector $(1,1,\dots,1)^T$ and $I$ represents the identity matrix; their size will be clear from the context. Given a vector $v$, the diagonal matrix $V$ is defined as $V=\text{diag}(v)$ and $v^k_j$ represents the $j$-th component of vector $v$ at the $k$-th iteration. The notation $v>0$ indicates that each component $v_j$ is strictly positive. $\|\cdot\|$ represents the Euclidean norm.

\section{Interior Point Method}
\label{section_ipm}
Consider a pair of primal-dual convex quadratic programming problems in standard form:
\begin{equation} \label{primal}
\underset{x}{\text{min}} \  c^T x+\frac{1}{2}x^TQx , \qquad \text{s.t.}  \  Ax = b, \qquad  \ x \geq 0, 
\end{equation}
\begin{equation} \label{dual}
\underset{y, \ s}{\text{max}}  \ b^Ty-\frac{1}{2}x^TQx , \qquad \text{s.t.}\   A^T y + s -Qx= c,\qquad s \geq 0,
\end{equation}
where $x,s,c\in\mathbb R^n$, $y,b\in\mathbb R^m$, $A\in\mathbb R^{m\times n}$, $Q\in\mathbb R^{n\times n}$ positive semidefinite.

An Interior Point Method (IPM) looks for an approximation of the solution to \eqref{primal}-\eqref{dual} in the interior of the feasible region; the non-negativity constraint is enforced using a logarithmic barrier term, so that the Lagrangian takes the form
\[\mathcal L(x,y,\mu)=c^Tx+\frac{1}{2}x^TQx-y^T(Ax-b)-\mu\sum_{i=1}^n\log(x_j).\]
The optimality conditions for this perturbed problem are 
\[\begin{cases}
Ax=b\\
A^Ty+s-Qx=c\\
XSe=\mu e\\
(x,s)\geq0.
\end{cases}
\]
The Newton method applied to the previous mildly nonlinear system of equations produces the following linear system, to be solved at each IPM iteration
\begin{equation}\label{unreduced_system}
\begin{bmatrix} A\ & 0\ & 0\\-Q\ & A^T\ & I\\ S\ & 0\ & X\end{bmatrix}\begin{bmatrix} \Delta x\\\Delta y\\\Delta s\end{bmatrix}=\begin{bmatrix}r_P\\r_D\\r_\mu\end{bmatrix}=\begin{bmatrix}b-Ax\\c+Qx-A^Ty-s\\\sigma\mu e-XSe\end{bmatrix},
\end{equation}
where $\sigma$ is the parameter responsible for the reduction in the complementarity measure $\mu=(x^Ts)/n$.

System \eqref{unreduced_system} is usually reduced to the augmented system
\begin{equation}\label{augsyst}
\begin{bmatrix} -Q-\Theta^{-1}\ & A^T\\ A\ & 0\end{bmatrix}\begin{bmatrix} \Delta x\\\Delta y\end{bmatrix}=\begin{bmatrix} r_D-X^{-1}r_\mu\\ r_P\end{bmatrix}
\end{equation}
where $\Theta=XS^{-1}$, and solved using an indefinite factorization or an iterative method for symmetric indefinite systems (e.g.\ MINRES \cite{minres}), or it is further reduced to the normal equations
\begin{equation}\label{normeq}
A(Q+\Theta^{-1})^{-1} A^T\Delta y=r_P+A(Q+\Theta^{-1})^{-1}(r_D-X^{-1}r_\mu)
\end{equation}
and solved with a Cholesky factorization or using the Conjugate Gradient (CG) method. The direction is then used to compute the stepsize $\alpha$ and to find the next point $(x+\alpha\Delta x,y+\alpha\Delta y,s+\alpha\Delta s)$. The outer iterations are stopped as soon as the approximation satisfies the following IPM stopping criterion
\begin{equation}\label{IPM_termination}
\frac{\|b-Ax\|}{\|b\|}\le\tau_P,\qquad\frac{\|c+Qx-A^Ty-s\|}{\|c\|}\le\tau_D,\qquad\mu\le\tau_\mu, 
\end{equation}
where $\tau_P$, $\tau_D$ and $\tau_\mu$ are predetermined tolerances.
In recent years a lot of effort has been put into designing efficient preconditioners for the augmented system and for the normal equations (e.g.\ \cite{ip_pmm,berga_indefinite,iter:BCO-COAP,mfipm,iter:OS-pcg1,iter:VOC}). The major difficulty originates from the extreme ill conditioning of the matrices in \eqref{augsyst} and \eqref{normeq} when $\mu$ gets close to zero; regularization is a common strategy employed to improve the conditioning of the problem (see \cite{quadreg,regulipm,nondiagreg}). 

The most successful implementations of the primal-dual Interior Point Methods are the path-following methods, where the approximations computed throughout the IPM iterations are forced to follow the central path and stay in the appropriately chosen neighbourhood of it. These algorithms show polynomial complexity, both in the feasible and infeasible case (see e.g.\ \cite{wright}). A common neighbourhood of the central path used in the infeasible case is defined as follows: at iteration $k$, the point $(x^k,y^k,s^k)$ is in the neighbourhood $\mathcal N_\infty(\gamma,\beta)$ if it satisfies
\begin{subequations}
\label{criterion}
\begin{equation}
\label{criterion_1}
(x^k,s^k)>0,
\end{equation}
\begin{equation}
\label{criterion_2}
\gamma\mu^k\le x^k_js^k_j\le\mu^k/\gamma,\quad\forall j,
\end{equation}
\begin{equation}
\label{criterion_3}
\|r_P^k\|\le\|r_P^0\|\beta\mu^k/\mu^0,\quad\|r_D^k\|\le\|r_D^0\|\beta\mu^k/\mu^0,
\end{equation}
\end{subequations}
where $0<\gamma<1$ and $\beta\ge1$ are two constants chosen at the beginning of the IPM algorithm. This paper focuses both on the augmented system approach \eqref{augsyst}, when dealing with generic QPs, and on the normal equations approach \eqref{normeq}, when dealing with LPs or special cases of QPs.

\section{Estimating the convergence of the outer iterations}
\label{section_ipcg}
This section provides a description of how to estimate the IPM convergence indicators throughout the CG or MINRES iterations; these quantities are used to terminate the linear iterations prematurely, without relying on the residual of the linear system. The main indicators that are commonly used, as shown in~\eqref{IPM_termination}, are the primal and dual infeasibilities and the complementarity gap. Algorithms and stopping criteria are derived both for CG and MINRES, to be used for LPs and QPs, respectively.

In the case of the normal equations for an LP, the CG is applied to system~\eqref{normeq} with $Q=0$; this means that at every inner iteration the approximation for $\Delta y$ is considered. In order to estimate the IPM indicators, $\Delta x$ and $\Delta s$ are also needed, which are computed as follows
\[\Delta x=(S^{-1}r_\mu-\Theta r_D)+\Theta A^T\Delta y,\]
\begin{equation}\label{builddeltas}
\Delta s=X^{-1}r_\mu-\Theta^{-1}\Delta x.\end{equation}
When using the augmented system instead, $\Delta x$ and $\Delta y$ are readily available and just $\Delta s$ needs to be computed.

These two formulas contain a first term, which is constant during the inner iterations, and a second term which varies as the Krylov method progresses. Once the full direction is known, the step to the boundary can be computed as
\begin{equation}\label{compute_alpha}\alpha_x^\text{max}=\min_{j\colon\Delta x_j<0} -\frac{x_j}{\Delta x_j},\quad\alpha_s^\text{max}=\min_{j\colon\Delta s_j<0} -\frac{s_j}{\Delta s_j}.\end{equation}
To determine primal and dual stepsizes, at each iteration a practical
IPM algorithm uses a fraction of the maximum step to the boundary.
It is computed as in \eqref{compute_alpha}, scaled by a certain factor (e.g. 0.995)
to guarantee that each point is in the interior of the feasible region. In this way the stepsizes are
\begin{equation}\label{compute_stepsize}
\alpha_x=0.995\,\alpha_x^\text{max},\quad\alpha_s=0.995\,\alpha_s^\text{max}.\end{equation}

\subsection{IPCG for LP}
Consider now the normal equations approach for an LP (i.e.\ $Q=0$).  Suppose the algorithm stops the CG at a certain iteration for which  the full direction $(\Delta x,\Delta y,\Delta s)$ and the stepsizes $\alpha_x$ and $\alpha_s$ have been computed. Let us indicate the new point by $(\bar x,\bar y,\bar s)$, then the infeasibilities can be written as
\[A\bar x-b=(Ax-b)+\alpha_x\Big((A\Theta A^T\Delta y)+(AS^{-1}r_\mu-A\Theta r_D)\Big),\]
\[A^T\bar y+\bar s-c=(A^Ty+s-c)+\alpha_s(A^T\Delta y+\Delta s).\]
The problematic terms in these formulas are given by $A^T\Delta y$ and $A\Theta A^T\Delta y$: computing these quantities at each linear iteration would require extra matrix operations to be performed. Define the vectors $v_1=X^{-1}r_\mu$, $v_2=S^{-1}r_\mu-\Theta r_D$, $v_3=AS^{-1}r_\mu-A\Theta r_D$, $\xi_1=A^T\Delta y$, $\xi_2=A\Theta A^T\Delta y$; then, the previous expressions become
\begin{equation}\label{ipcg_directions}\Delta x=v_2+\Theta\xi_1,\qquad\Delta s=v_1-\Theta^{-1}\Delta x,\end{equation}
\begin{equation}\label{ipcg_primal}A\bar x-b=(Ax-b)+\alpha_x(\xi_2+v_3),\end{equation}
\begin{equation}\label{ipcg_dual}A^T\bar y+\bar s-c=(A^Ty+s-c)+\alpha_s(\xi_1+\Delta s).\end{equation}
Vectors $v_1$, $v_2$ and $v_3$ remain constant during the CG iterations and can be computed once at the beginning of the algorithm. Recall that, during the CG process, the approximation $\Delta y$ is updated as
\[\Delta y\leftarrow\Delta y+\alpha^{CG} u\]
where $\alpha^{CG}$ is the CG stepsize and $u$ is the CG direction. Therefore, the quantities $\xi_1$ and $\xi_2$ can be updated in a similar way:
\[\xi_1\leftarrow\xi_1+\alpha^{CG} A^Tu,\qquad \xi_2\leftarrow\xi_2+\alpha^{CG} A\Theta A^Tu.\]
The quantity $A\Theta A^Tu$ is already computed during the CG algorithm, because it is needed to find the stepsize $\alpha^{CG}$ and to update the residual. While computing it, the quantity $A^Tu$ can be obtained as a byproduct:
\[w_1=A^Tu,\qquad w_2=A\Theta w_1.\]
In this way, it is possible to update the quantities $\xi_1$ and $\xi_2$ at each inner iteration inexpensively, which in turn allows to compute the IPM convergence indicators at each CG iteration using only vector operations. Notice that the products with matrix $\Theta$ needed to compute the directions in~\eqref{ipcg_directions}, in practice are performed as vector operations, since $\Theta$ is diagonal. Notice also that $w_1$ and $w_2$ need to be computed at the beginning of the CG process, to initialize the residual; thus, initializing $\xi_1$ and $\xi_2$  does not add operations. However, one single matrix-vector product with matrix $A$ is added at the beginning of the algorithm, to compute the constant vector $v_3$. Algorithm~\ref{ipcg_alg} summarizes the process just described: the main differences with the standard CG algorithm are in lines \ref{line1}, \ref{line2}, \ref{line3}, \ref{line4}, \ref{line5}, \ref{line6}. The IPM convergence indicators are estimated only after a number \texttt{itstart} of iterations. The algorithm does not contain any stopping criterion for now, it simply computes the primal and dual infeasibilities and the duality gap at each CG iteration, if the CG process was stopped at that iteration. Other indicators can also be computed if needed, as will be clear in the next sections. The choice of the stopping criterion based on these indicators will be discussed in the following sections.

At each iteration, the standard CG algorithm performs one matrix-vector product, one preconditioner application, two scalar products and three {\tt axpy} operations; what the authors propose to add in the IPCG algorithm requires, at each iteration, the equivalent of three scalar products (to compute $\mu$, $\Theta\xi_1$, $\Theta^{-1}\Delta x$), approximately ten {\tt axpy} operations and the computation of the stepsizes (which are computed as in \eqref{compute_alpha}-\eqref{compute_stepsize} and thus involve only comparison of vector components and component-wise divisions). Therefore, the authors expect the computational cost of the IPCG iteration to be only slightly larger than that of the standard CG step, especially if the applications of the matrix or the preconditioner are particularly expensive. 

\renewcommand{\thealgorithm}{IPCG}
\begin{algorithm}[h!]
\footnotesize
\caption{Interior Point Conjugate Gradient method}
    \label{ipcg_alg}
    \textbf{Input:} rhs $f$, tolerance $\tau_\text{inner}$, max iterations {\tt itmax}, matrices $A,\Theta$, preconditioner $P$, initial approximation $\Delta y$, minimum iterations {\tt itstart}

   \textbf{Input from IPM: } current point $(x,y,s)$, vectors $r_P,r_D,r_\mu$

\vspace{5pt}
\begin{algorithmic}[1]
\State \textbf{Initialize:}
\State $v_1=X^{-1}r_\mu$, $v_2=\Theta(v_1-r_D)$, $v_3=Av_2$\label{line1}
\State $\xi_1=A^T\Delta y$
\State $\xi_2=A\Theta\xi_1$
\State $r_0=f-\xi_2$
\State $r=r_0$
\State $z=P^{-1}r$
\State $u=z$
\State $\rho=r^Tz$
\State $\mathtt{iter}=0$

\While{$\|r\|>\tau_\text{inner}\|r_0\|$ and $\mathtt{iter}<\mathtt{itmax}$}
\State $\mathtt{iter}=\mathtt{iter}+1$
\State $w_1=A^Tu$
\State$w_2=A\Theta w_1$
\State $\alpha^{CG}=\rho/w_2^Tu$
\State $\Delta y=\Delta y+\alpha^{CG} u$
\State $\xi_1=\xi_1+\alpha^{CG} w_1$\label{line2}
\State $\xi_2=\xi_2+\alpha^{CG} w_2$\label{line3}
\State $r=r-\alpha^{CG} w_2$
\State $z=P^{-1}r$
\State $\rho^N=r^Tz$
\State $\beta=\rho^N/\rho$
\State $u=z+\beta u$
\State $\rho=\rho^N$
\If {($\mathtt{iter}\ge\mathtt{itstart}$)}
\State Compute Newton directions: $\Delta x=v_2+\Theta\xi_1,\quad\Delta s=v_1-\Theta^{-1}\Delta x$\label{line4}
\State Compute stepsizes $\alpha_x$, $\alpha_s$ using $x,\,\Delta x,\,s,\,\Delta s$\label{line5}
\State Compute convergence indicators:\label{line6}
\[p_\text{inf}=-r_P+\alpha_x(\xi_2+v_3),\quad d_\text{inf}=-r_D+\alpha_s(\xi_1+\Delta s),\quad\mu=(x+\alpha_x\Delta x)^T(s+\alpha_s\Delta s)\]
\EndIf
\EndWhile
\end{algorithmic}
\end{algorithm}

\begin{remark}
Notice that, when using a predictor-corrector strategy, the algorithm just proposed works only when computing the predictor direction. For the corrector, equations~\eqref{ipcg_primal}-\eqref{ipcg_dual} need to be modified. In particular, calling $(\Delta x^P,\Delta y^P,\Delta s^P)$ the predictor computed previously, the terms to add are $A\Delta x^P$ to the expression for the primal residual~\eqref{ipcg_primal} and $A^T\Delta y^P+\Delta s^P$ to the expression for the dual residual~\eqref{ipcg_dual}; they can be computed at the beginning since they are constant, but they add matrix operations to be performed at every call of the algorithm. Alternatively, these can be avoided by saving the final values of the vectors $\xi_1$ and $\xi_2$ from the previous IPCG call that computed the predictor direction.
\end{remark}

\subsection{IPMINRES for QP}
\renewcommand{\thealgorithm}{IPMINRES}
\begin{algorithm}[h!]
\footnotesize
\caption{Interior Point Minimum Residual method}
    \label{ipminres_alg}
\textbf{Input:} rhs $f$, tolerance $\tau_\text{inner}$, max iterations {\tt itmax}, matrices $A,\Theta$,Q, preconditioner $P$, minimum iterations {\tt itstart}

\textbf{Input from IPM: } current point $(x,y,s)$, vectors $r_P,r_D,r_\mu$

\vspace{5pt}
\begin{spacing}{1.1}
\begin{algorithmic}[1]
\State \textbf{Initialize:}
\State $\psi=P^{-1}f$, $r_1=f$, $r_2=r_1$, $\beta=\sqrt{f^T\psi}$, $w=0$, $w_2=0$, $c_s=-1$, $s_n=0$, $\bar\varphi=\beta$, $\epsilon=0$, $\Delta=0$, $\texttt{iter}=0$ 
\State $w^v=0$, $w^v_2=0$, $\xi=0$, $\zeta=X^{-1}r_\mu$\label{linem1}
\While{$\texttt{residual}>\tau_\text{inner}\|\texttt{residual}_0\|$ and $\texttt{iter}<\texttt{itmax}$}
\State $\texttt{iter}=\texttt{iter}+1$
\State $v=\begin{bmatrix}v_1\\v_2\end{bmatrix}=\dfrac{1}{\beta}\psi$
\State $\psi=\begin{bmatrix}-Q-\Theta^{-1}\ & A^T\\A\ & 0\end{bmatrix}\begin{bmatrix}v_1\\v_2\end{bmatrix}$, with byproduct $z^v=\begin{bmatrix}Qv_1\\Av_1\\A^Tv_2\end{bmatrix}$\label{linem2}
\State \textbf{if} $\texttt{iter}\ge2$ \textbf{then} $\psi=\psi-(\beta/\beta_0) r_1$ \textbf{end if}
\State $\alpha=v^T\psi$
\State $\psi=\psi-(\alpha/\beta)r_2$
\State $r_1=r_2$, $r_2=\psi$
\State $\psi=P^{-1}r_2$
\State $\beta_0=\beta$, $\beta=\sqrt{r_2^T\psi}$
\State $\epsilon_0=\epsilon$, $\delta=c_s\bar\delta+s_n\alpha$, $\bar g=s_n\bar\delta-c_s\alpha$, $\epsilon=s_n\beta$, $\bar\delta=-c_s\beta$, $r=\sqrt{\bar g^2+\bar\delta^2}$
\State $\gamma=\max(\sqrt{\bar g^2+\beta^2},\epsilon)$
\State $c_s=\bar g/\gamma$, $s_n=\beta/\gamma$, $\varphi=c_s\bar\varphi$, $\bar\varphi=s_n\bar\varphi$
\State $w_1=w_2$, $w_2=w$, $w^v_1=w^v_2$, $w^v_2=w^v$\label{linem3}
\State $w=(v-\epsilon_0w_1-\delta w_2)/\gamma$
\State $w^v=(z^v-\epsilon_0 w^v_1-\delta w^v_2)/\gamma$\label{linem4}
\State $\Delta=\begin{bmatrix}\Delta x\\\Delta y\end{bmatrix}=\Delta+\varphi w$
\State $\xi=\begin{bmatrix}\xi_Q\\\xi_x\\\xi_y\end{bmatrix}=\xi+\varphi w^v$\label{linem5}
\If {($\mathtt{iter}\ge\mathtt{itstart}$)}
\State Compute Newton direction: $\Delta s=\zeta-\Theta^{-1}\Delta x$\label{linem6}
\State Compute stepsizes $\alpha_x$, $\alpha_s$ using $x,\,\Delta x,\,s,\,\Delta s$\label{linem7}
\State Compute convergence indicators:\label{linem8}
\[p_\text{inf}=-r_P+\alpha_x\xi_x,\quad d_\text{inf}=-r_D+\alpha_s(\xi_y+\Delta s)-\alpha_x\xi_Q,\quad\mu=(x+\alpha_x\Delta x)^T(s+\alpha_s\Delta s)\]
\EndIf
\EndWhile
\end{algorithmic}
\end{spacing}
\end{algorithm}
Similarly, in the case of the augmented system for a QP, the infeasibilities can be written as
\[A\bar x-b=(Ax-b)+\alpha_x(A\Delta x)\]
\[A^T\bar y+\bar s-Q\bar x-c=(A^Ty+s-Qx-c)+\alpha_s(A^T\Delta y+\Delta s)-\alpha_x(Q\Delta x).\]
Thus, at each inner iteration, the quantities to update are $\xi_x=A\Delta x$, $\xi_y=A^T\Delta y$, $\xi_Q=Q\Delta x$; this can be done at little extra cost by exploiting the matrix-vector products already present in the MINRES algorithm, similarly to what was  done earlier for the CG. The implementation is slightly more complicated, since the MINRES updates the approximation using the two previous iterations; Algorithm~\ref{ipminres_alg} shows the standard MINRES method, according to the implementation in \cite{minres_web}, with the additional operations required: the main differences with the standard MINRES algorithm are in lines \ref{linem1}, \ref{linem2}, \ref{linem3}, \ref{linem4}, \ref{linem5}, \ref{linem6}, \ref{linem7}, \ref{linem8}. The estimation of the residual is more complicated than in the CG case and it is not shown to avoid further overcomplicating of the displayed algorithm and because it is not affected by the new approach. As before, the additional cost is given only by vector operations (scalar products, {\tt axpy} operations and stepsizes computation).


\begin{remark}
The letter $\alpha$ has been used to indicate multiple stepsizes related to IPM and CG. Since these are the standard notations in both fields, the authors did not change them but added subscripts and superscripts to identify them ($\alpha_x$ and $\alpha_s$ for the IPM stepsizes, while $\alpha^{CG}$ for the CG stepsize).
\end{remark}

\section{Stopping criterion}
\label{section_stop}

The next section provides some arguments for the complexity analysis of the proposed method. Due to the complex interaction between the IPM and the linear solvers, a full polynomial complexity proof is not presented; instead, a rationale is provided which suggests that the proposed inexact algorithm should have only slightly weaker complexity than the exact method.


\subsection{Complexity analysis}
This section will follow \cite[chapter 6]{wright}, with the difference that here the problem is a quadratic program. The authors make some standard assumptions: the relative interior of problems \eqref{primal}-\eqref{dual} is non-empty; the neighbourhood is defined by~\eqref{criterion}; the parameter $\sigma_k$ is chosen in the interval $[\sigma_\text{min},\sigma_\text{max}]$, $\sigma_\text{max}\le1$; a single stepsize $\alpha_k$ is considered instead of two different ones for the primal and dual direction; the stepsize is chosen such that the next point is inside the central path neighbourhood and it satisfies the Armijo condition
\begin{equation}\label{armijo}
\mu_{k+1}\le(1-0.01\alpha_k)\mu_k.\end{equation}
It is already known that, when dealing with an LP and using an exact method to find the direction, there is a minimum stepsize that can be taken, $\alpha_k\ge\bar\alpha$, and both the primal and dual infeasibilities are reduced by a factor $(1-\alpha_k)$. Moreover, the third equation in~\eqref{unreduced_system} yields
\begin{equation}\label{thirdequation_order1}\frac{\Delta x^k_j}{x^k_j}+\frac{\Delta s^k_j}{s^k_j}=-1+\frac{\sigma_k\mu_k}{x^k_js^k_j},\quad\forall j.\end{equation}
Notice that the right hand side in the last equation is $\mathcal O(1)$, due to the symmetric neighbourhood used \eqref{criterion_2}. Therefore, given these facts, the stopping criterion is chosen as follows: the direction produced by the inner solver is accepted as soon as
\begin{equation}\label{ipcg_criterion_1}
\max_j \Big|\frac{\Delta x^k_j}{x^k_j}\Big|\le M,\qquad\max_j \Big|\frac{\Delta s^k_j}{s^k_j}\Big|\le M\end{equation}
for some fixed constant $M$. Moreover, it should also be required that
\begin{equation}\label{ipcg_criterion_2}
\|r_P^{k+1}\|\le\eta_{k+1}\|r_P^k\|,\quad\|r_D^{k+1}\|\le\eta_{k+1}\|r_D^k\|,\end{equation}
where $\eta_k\ge1-\alpha_{k-1}$, since an inexact direction cannot reasonably perform as well as the exact one. Thus, suppose that $\eta_k=1-\omega_k\alpha_{k-1}$, for some $\omega_k\le1$; the choice of $\omega_k$ will be clarified in the next Lemma. Suppose also that the equation 
\begin{equation}\label{third_equation}S^k\Delta x^k+X^k\Delta s^k=\sigma_k\mu_k e-X^kS^ke\end{equation}
 continues to hold even if the direction is inexact;  this is the case if $\Delta s$ is built from $\Delta x$, as it usually happens when employing the normal equations or the augmented system, as shown in~\eqref{builddeltas}. Algorithm~\ref{ipm_alg} summarizes the choices made here and shows also some other features that will be clear during the proof of Lemma~\ref{lemma_alpha_bound}. 
 
\begin{remark}
In the algorithm, it may seem like the stepsize $\alpha_k$ is used at step \ref{ipmi_ipcg} but is computed only at step \ref{ipmi_alpha}. This happens because the stepsize is estimated at every inner iteration, as shown in Algorithms \ref{ipcg_alg} and \ref{ipminres_alg}. To avoid confusion, the estimated stepsize in step \ref{ipmi_ipcg} is denoted as $\hat\alpha_k$.
\end{remark}

\renewcommand{\thealgorithm}{IPM-I}
\begin{algorithm}[h]
\footnotesize
\caption{Interior Point Method with early stopping of the linear solver}
    \label{ipm_alg}
    \textbf{Input:} $\gamma\in[0,1]$, $\beta\ge1$, $0<\delta<\sigma_\text{min}<\sigma_\text{max}\le1$, $M>0$

\begin{algorithmic}[1]
\State Choose $(x^0,y^0,s^0)$ with $(x^0,s^0)>0$
\While{\eqref{IPM_termination} is not satisfied}
\State Choose $\sigma_k\in[\sigma_\text{min},\sigma_\text{max}]$
\State Choose $\omega_k\in[1-\sigma_k+\delta,1]$
\State Compute $\mu_k=(x^k)^Ts^k/n$
\State Find a direction $(\Delta x^k,\Delta y^k,\Delta s^k)$ such that \label{ipmi_ipcg}
	\[\max_j \Big|\frac{\Delta x^k_j}{x^k_j}\Big|\le M,\qquad\max_j \Big|\frac{\Delta s^k_j}{s^k_j}\Big|\le M,\]
	\[\|r_P^k\|\le(1-\omega_k\hat\alpha_k)\|r_P^{k-1}\|,\quad\|r_D^k\|\le(1-\omega_k\hat\alpha_k)\|r_D^{k-1}\|,\]
	\[S^k\Delta x^k+X^k\Delta s^k=\sigma_k\mu_k e-X^kS^ke.\]
\State Choose $\alpha_k$ as the largest $\alpha\in[0,1]$ such that \label{ipmi_alpha}
	\[(x^k+\alpha\Delta x^k,y^k+\alpha\Delta y^k,s^k+\alpha\Delta s^k)\in\mathcal N_\infty(\gamma,\beta),\]
	\[(x^k+\alpha\Delta x^k)^T(s^k+\alpha\Delta s^k)\le(1-0.01\alpha)(x^k)^T(s^k).\]
\State Set 
\[(x^{k+1},y^{k+1},s^{k+1})=(x^k+\alpha_k\Delta x^k,y^k+\alpha_k\Delta y^k,s^k+\alpha_k\Delta s^k)\]
\EndWhile
\end{algorithmic}
\end{algorithm}

The Lemma below asserts that, if the direction is chosen using the stopping criterion defined by~\eqref{ipcg_criterion_1}-\eqref{ipcg_criterion_2}, then there exists a minimum stepsize $\tilde\alpha$ such that a new iterate after a step in the (inexact) Newton direction belongs to the $\mathcal N_\infty(\gamma,\beta)$ neighbourhood and delivers a guaranteed reduction of the complementarity product. In the following, the iteration index $k$ will be omitted, for sake of clarity.

\begin{lemma}
\label{lemma_alpha_bound} 
Consider an IPM algorithm where each direction satisfies conditions \eqref{ipcg_criterion_1}-\eqref{ipcg_criterion_2}-\eqref{third_equation}. Suppose that $\omega$ is chosen such that $\omega\ge1-\sigma+\delta$, where $\delta<\sigma_\text{min}$ is a constant.

Then, there exists a value $\tilde\alpha\in(0,1)$ such that the following conditions are satisfied for all $\alpha\in[0,\tilde\alpha]$ at each IPM iteration and for all components $j$:
\begin{subequations}
\begin{equation}\label{proof_ineq_1}(x_j+\alpha\Delta x_j)(s_j+\alpha\Delta s_j)\ge\gamma(x+\alpha\Delta x)^T(s+\alpha\Delta s)/n,\end{equation}
\begin{equation}\label{proof_ineq_2}(x_j+\alpha\Delta x_j)(s_j+\alpha\Delta s_j)\le(1/\gamma)(x+\alpha\Delta x)^T(s+\alpha\Delta s)/n,\end{equation}
\begin{equation}\label{proof_ineq_3}(x+\alpha\Delta x)^T(s+\alpha\Delta s)/n\le(1-0.01\alpha)\mu,\end{equation}
\begin{equation}\label{proof_ineq_4}(x+\alpha\Delta x)^T(s+\alpha\Delta s)\ge\eta x^Ts.\end{equation}
\end{subequations}
In particular, the minimum stepsize is
\begin{equation}\label{alpha_bound}
\tilde\alpha=\min\Big(\frac{\sigma_\text{min}\gamma(1-\gamma)}{M^2(1+\gamma^2)},\frac{\sigma_\text{min}(1-\gamma)}{2M^2},\frac{0.99-\sigma_\text{max}}{M^2},\frac{\delta}{M^2},1\Big).\end{equation}

\end{lemma}
\begin{proof}
Start by noticing that~\eqref{ipcg_criterion_1} implies these two facts:
\begin{enumerate}
\item the positivity constraints $x+\alpha\Delta x>0$ and $s+\alpha\Delta s>0$ are automatically satisfied for any $\alpha\in[0,\frac{1}{M}[$;
\item the following bounds hold
\begin{equation}\label{proof_bound}|\Delta x_j\Delta s_j|\le \frac{M^2}{\gamma}\mu,\qquad |\Delta x^T\Delta s|\le M^2n\mu.\end{equation}
\end{enumerate}
Using \eqref{proof_bound}, \eqref{third_equation} and \eqref{criterion_2}, it is easy to show that the following inequalities hold
\begin{subequations}\label{ineq_proof}
\begin{equation}\label{ineq_1}(x_j+\alpha\Delta x_j)(s_j+\alpha\Delta s_j)\ge(1-\alpha)\gamma\mu+\alpha\sigma\mu-\alpha^2M^2\mu/\gamma,\end{equation}
\begin{equation}\label{ineq_2}(x_j+\alpha\Delta x_j)(s_j+\alpha\Delta s_j)\le(1-\alpha)\mu/\gamma+\alpha\sigma\mu+\alpha^2M^2\mu/\gamma,\end{equation}
\begin{equation}\label{ineq_3}(x+\alpha\Delta x)^T(s+\alpha\Delta s)/n\ge(1-\alpha)\mu+\alpha\sigma\mu-\alpha^2M^2\mu,\end{equation}
\begin{equation}\label{ineq_4}(x+\alpha\Delta x)^T(s+\alpha\Delta s)/n\le(1-\alpha)\mu+\alpha\sigma\mu+\alpha^2M^2\mu.\end{equation}
\end{subequations}

\noindent Using \eqref{ineq_1} and \eqref{ineq_4}, it follows that
\begin{align}&(x_j+\alpha\Delta x_j)(s_j+\alpha\Delta s_j)-\gamma(x+\alpha\Delta x)^T(s+\alpha\Delta s)/n\ge\notag\\
&\ge(1-\alpha)\gamma\mu+\alpha\sigma\mu-\alpha^2M^2\frac{\mu}{\gamma}-\gamma((1-\alpha)\mu+\alpha\sigma\mu+\alpha^2M^2\mu)\ge\notag\\
&\ge\alpha\sigma_\text{min}\mu(1-\gamma)-\alpha^2M^2\mu(\gamma+1/\gamma)\notag
\end{align}
and thus \eqref{proof_ineq_1} is satisfied if the final expression is non-negative, i.e.\
\[\alpha\le\dfrac{\sigma_\text{min}\gamma(1-\gamma)}{M^2(1+\gamma^2)}.\]

\noindent In a similar way, using \eqref{ineq_2} and \eqref{ineq_3}, it can be shown that \eqref{proof_ineq_2} is satisfied if
\[\alpha\le\dfrac{\sigma_\text{min}(1-\gamma)}{2M^2},\]
and using \eqref{ineq_4}, \eqref{proof_ineq_3} is satisfied if
\[\alpha\le\dfrac{0.99-\sigma_\text{max}}{M^2}.\]


\noindent Using \eqref{ineq_3} and setting $\eta = 1 - \omega \alpha$, it follows that
\begin{align}
&(x+\alpha\Delta x)^T(s+\alpha\Delta s)-(1-\omega\alpha)x^Ts\ge\notag\\
&\ge(1-\alpha)x^Ts+\alpha\sigma x^Ts-\alpha^2M^2x^Ts-(1-\omega\alpha)x^Ts\ge\notag\\
&\ge\alpha(\sigma+\omega-1)x^Ts-\alpha^2M^2x^Ts\notag
\end{align}
and thus \eqref{proof_ineq_4} is satisfied if the final expression is non-negative, i.e.\
\[\alpha\le\dfrac{\sigma+\omega-1}{M^2}.\]
This condition makes sense only if $\omega>1-\sigma$; therefore, at each IPM iteration, after choosing $\sigma$, $\omega$ should be chosen from the interval $]1-\sigma,1]$. Notice what this means: in the early IPM iterations, $\sigma$ is closer to 1 and thus $\omega$ can be closer to 0, which makes the stop criterion easier to satisfy. In the later iterations, $\sigma$ might get closer to 0 and thus $\omega$ is closer to 1, which makes the stop criterion harder to satisfy. If $\omega$ is chosen such that $\omega\ge1-\sigma+\delta$, with $\delta$ a fixed constant, $\delta<\sigma_\text{min}$, then $\omega+\sigma-1\ge\delta$ and it follows that
\[\alpha\le\dfrac{\delta}{M^2}\quad\Rightarrow\quad\alpha\le\frac{\sigma+\omega-1}{M^2}.\]
This explains the choice of $\omega$ made in the statement of the Lemma.

Therefore, the minimum stepsize that can be taken at each IPM iteration is given by
\[\tilde\alpha=\min\Big(\frac{\sigma_\text{min}\gamma(1-\gamma)}{M^2(1+\gamma^2)},\frac{\sigma_\text{min}(1-\gamma)}{2M^2},\frac{0.99-\sigma_\text{max}}{M^2},\frac{\delta}{M^2},1\Big).\]

\end{proof}

Notice that inequalities \eqref{proof_ineq_1}-\eqref{proof_ineq_2} imply that the next IPM iteration satisfies condition~\eqref{criterion_2}; the inequality~\eqref{proof_ineq_3} represents the Armijo condition, while inequality~\eqref{proof_ineq_4} implies that
\[\frac{\|r_P^k\|}{\mu^k}\le\frac{\eta_k\|r_P^{k-1}\|}{\mu^k}\le\frac{\|r_P^{k-1}\|}{\mu^{k-1}}\le\frac{\beta\|r_P^0\|}{\mu^0}\]
and similarly for the dual residual, which is equivalent to condition~\eqref{criterion_3}. 

To obtain a polynomial complexity result, the value of $M$ should be specified as a function of $n$. Here, the complexity analysis becomes problematic, since it is difficult to determine exactly the properties of the IPM directions at intermediate Krylov iterations. This is the subject of further research, but for now a rationale is given, based on the properties of the exact directions. To start, recall the results in \cite[Chapter 6]{wright} about convergence of LPs (similar results for QPs can be found in \cite{wright_qp}): a minimum stepsize, proportional to $n^{-2}$ can be found at each iteration, provided that the starting point is chosen appropriately. In the following, this result is generalized to a generic starting point, under some mild assumptions.

\begin{lemma}
Suppose that the optimal solution $(x^*,s^*)$ satisfies $0\le x^*_i,s^*_i\le\xi$, for some large constant $\xi$. Given any starting point $(x^0,y^0,s^0)\in\mathcal N_\infty(\gamma,\beta)$ such that $0<x_i^0,s_i^0\le\xi\ \forall i$ and $\mu^0>\varepsilon^*>0$, the minimum step size for an exact IPM applied to an LP is $\bar\alpha\ge C_3/n^3$, for some positive constant $C_3$ independent of $n$.
\end{lemma}
\begin{proof}
Recall the following constants from \cite[Lemma 6.3 and 6.5]{wright}, used to find the minimum stepsize for an LP:
\[C_1=\Big(\beta n+n+\beta\frac{\text{max}^0}{\mu^0}\|(x^*,s^*)\|_1\Big)\frac{1}{\text{min}^0},\]
\[C_2=2\frac{C_1}{\gamma^{1/2}}\max\big(\|x^0-x^*\|,\|s^0-s^*\|)+\frac{n}{\gamma^{1/2}},\]
where $(x^*,s^*)$ is the optimal solution, $\text{min}^0$ and $\text{max}^0$ are the minimum and maximum components, respectively, of the vector $(x^0,s^0)$. Here, $(x,s)$ indicates the vector obtained stacking vertically the vectors $x$ and $s$.

From the definition of the neighbourhood \eqref{criterion} and the hypothesis used, it follows that, for each component $i$, $x^0_is^0_i\ge\gamma\mu^0>\gamma\varepsilon^*$ and thus
\[x^0_i>\frac{\gamma\varepsilon^*}{s^0_i}\ge\frac{\gamma\varepsilon^*}{\xi}.\]
Therefore $x_i^0, s_i^0\in[\gamma\varepsilon^*/\xi,\xi],\,\,\forall i$. Hence $\text{min}^0\ge\gamma\varepsilon^*/\xi$ and $\text{max}^0\le\xi$. Notice also that $\|(x^*,s^*)\|_1\le2n\xi$. Therefore
\[C_1\le\Big(\beta n+n+2\beta\frac{\xi^2}{\varepsilon^*}n\Big)\frac{\xi}{\gamma\varepsilon^*}.\]
Given that $x_i^0,s_i^0,x_i^*,s_i^*$ all belong to the interval $[0,\xi]$, for each component $i$, it follows that
\[\|x^0-x^*\|^2=\sum_{i=1}^n(x^0_i-x^*_i)^2\le\sum_{i=1}^n\xi^2=\xi^2n,\]
and the same holds for $\|s^0-s^*\|$. Therefore
\[C_2\le n^{3/2}\Big(\frac{2\xi^2}{\gamma^{3/2}\varepsilon^*}\Big)\Big(\beta+1+2\frac{\beta\xi^2}{\varepsilon^*}\Big)+\frac{n}{\gamma^{1/2}}\]
which implies $C_2\le\mathcal O(n^{3/2})$. The minimum stepsize $\bar\alpha$ that can be taken at each iteration in the exact IPM is proportional to $C_2^{-2}$ as shown in \cite[Lemma 6.7]{wright}, thus $\bar\alpha\ge C_3n^{-3}$.
\end{proof}

Therefore, if the starting point is not the optimal one indicated in \cite{wright}, the minimum stepsize is proportional to $n^{-3}$, instead of $n^{-2}$.  Consider the IPM termination criterion \eqref{IPM_termination} and property \eqref{criterion_3}; then, the algorithm converges if
\[\mu\le\min\Big(\tau_\mu,\tau_P\|b\|\frac{\mu^0}{\beta\|r_P^0\|},\tau_D\|c\|\frac{\mu^0}{\beta\|r_D^0\|}\Big)=\colon\varepsilon^*.\]
Therefore, assuming that the iterates of the {\it inexact} algorithm are bounded by $\xi$, at each iteration two things can happen: if $\mu\le\varepsilon^*$, then the algorithm converged; if $\mu>\varepsilon^*$, then the current point can be seen as a starting point of an exact IPM where the minimum stepsize is $\bar\alpha\ge C_3n^{-3}$.

Given that the stepsize must be strictly smaller than the step to the boundary \eqref{compute_alpha}, it is immediate to see that, when using an {\it exact} direction, the following holds
\[\frac{-\Delta x_j}{x_j}\le\frac{n^3}{C_3}\quad\forall j\,\,\text{s.t.}\,\,\Delta x_j<0,\qquad\frac{-\Delta s_j}{s_j}\le\frac{n^3}{C_3}\quad\forall j\,\,\text{s.t.}\,\,\Delta s_j<0.\]
Additionally, using \eqref{thirdequation_order1}, one can see that
\[\Big|\frac{\Delta x_j}{x_j}\Big|\le \mathcal O(n^3),\qquad \Big|\frac{\Delta s_j}{s_j}\Big|\le \mathcal O(n^3),\qquad\forall j,\]
since the terms $\frac{\Delta x_j}{x_j}$ and $\frac{\Delta s_j}{s_j}$ must balance in order to give a sum that is $\mathcal O(1)$.

Therefore, when using an {\it exact} IPM for LPs, with a generic starting point, the computed direction satisfies criterion \eqref{ipcg_criterion_1} with $M=\mathcal O(n^3)$. This provides a rationale to expect that {\it inexact} steps applied should satisfy conditions like \eqref{ipcg_criterion_1} with a constant $M$ of comparable magnitude. Therefore, the authors infer that it is possible to use a constant $M=\mathcal O(n^q)$, with $q\ge3$. This is of course only a rationale: a proper proof would require to understand whether the chosen Krylov method is able to deliver such a direction; potentially, an exponent $q$ specific to the linear solver used may be found, but this has shown to be complicated and is the subject of further research. Notice also that the rationale argument is given for an LP, but the criterion is used for QPs (similar arguments can be found for QPs, based on the results in \cite{wright_qp}).

Given $M=\mathcal O(n^q)$, \cite[Theorem 3.2]{wright} and Lemma~\ref{lemma_alpha_bound} imply that the number of iterations to achieve a $\nu-$accurate solution would be $\mathcal O(n^{2q}|\log\nu|)$. In the best case where $q=3$, this would mean a number of iterations proportional to $n^6$; this is higher than the $\mathcal O(n^2)$ iterations required by the exact algorithm, as it is to be expected from the very inexact stopping criterion considered.

\begin{remark}
The analysis presented in this section has used the results from \cite[Chapter 6]{wright}; it is worth pointing out that the results presented there are obtained using a neighbourhood without the upper bound in~\eqref{criterion_2}. However, with some simple calculations, it is possible to see that the final results do not change after adding the upper bound. A similar conclusion was obtained in \cite{colombo_gondzio}, where the upper bound was added in the case of a feasible algorithm.
\end{remark}

\subsection{Indicators for early stopping}
In this section, new indicators are derived that can be used to terminate the inner linear iterations early, before the relative residual has become small enough to be accepted by a standard residual test. The behaviour is shown for one of the test problems that are presented later, but the same pattern can be observed also for the other test problems. This specific problem is a QP without linear constraints that arises from tomographic imaging; thus, the normal equations approach is used, but only the dual infeasibility can be defined.

The indicators that are introduced are based on the complexity argument given in the previous section; they are related to the following quantities:
\begin{itemize}
\item $M_x=\max_i |\frac{\Delta x_i}{x_i}|$ and $M_s=\max_i |\frac{\Delta s_i}{s_i}|$
\item infeasibilities: $p_\text{inf}=b-Ax$ and $d_\text{inf}=c-A^Ty-s$
\item complementarity measure: $\mu=(x^Ts)/n$.
\end{itemize}
These are computed at each inner Krylov iteration using the IPCG or IPMINRES algorithms; an index $j$ is used to indicate the value obtained at the inner iteration $j$ and the IPM iteration index is omitted instead. Therefore, $M_x^j$ means the value of the quantity $M_x$ that would be obtained by stopping at the $j-$th inner iteration, for a given IPM iteration.

Figure~\ref{fig_ipm_iter} displays the behaviour of the dual infeasibility, complementarity, primal and dual stepsizes and the quantities $M_x$ and $M_s$ at an intermediate IPM iteration; they are computed at every inner CG iteration, using Algorithm~\ref{ipcg_alg}.

\begin{figure}[h]
\caption{Infeasibility, complementarity, stepsizes and quantities $M_x$ and $M_s$ computed at every CG iteration, for an intermediate IPM iteration.}
\label{fig_ipm_iter}
\centering
\includegraphics[width=\textwidth]{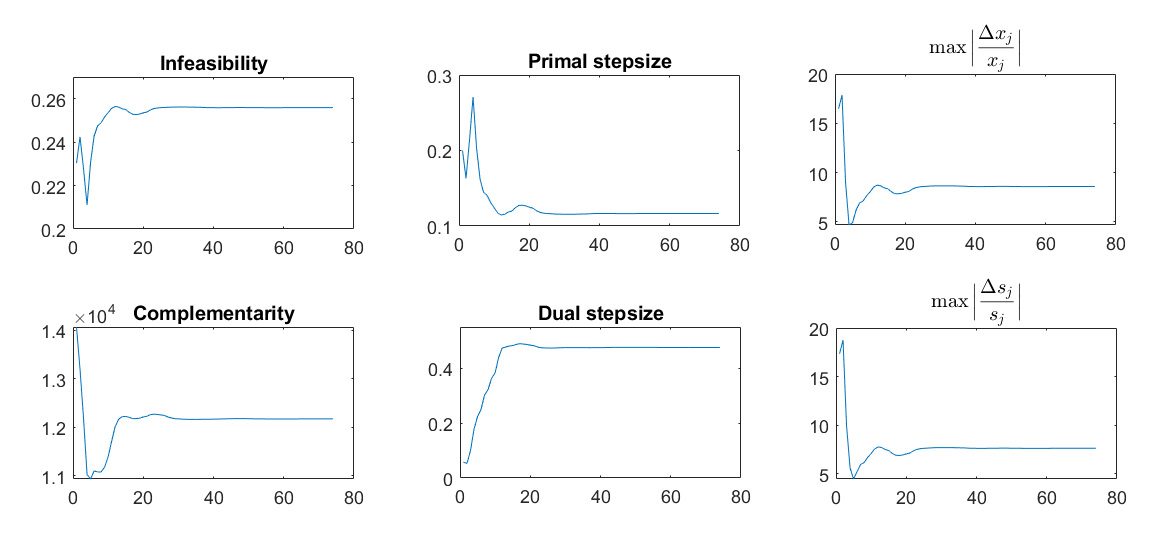}
\end{figure}

It can be seen that all the quantities represented reach a point where their variation becomes extremely small, almost impossible to notice from the picture; this "stagnation" point may arrive very early in the CG iterations, meaning that a large portion of the inner iterations are used to adjust the IPM direction in a way that has a small effect on the quality of the new IPM point.

This fact suggests the following early termination indicators:
\[\texttt{var}_P^j=\frac{1}{5}\sum_{i=0}^4\Biggl|\frac{\|p_\text{inf}^{j-i}\|-\|p_\text{inf}^{j-i-1}\|}{\|p_\text{inf}^{j-i-1}\|}\Biggr|,\quad\texttt{var}_D^j=\frac{1}{5}\sum_{i=0}^4\Biggl|\frac{\|d_\text{inf}^{j-i}\|-\|d_\text{inf}^{j-i-1}\|}{\|d_\text{inf}^{j-i-1}\|}\Biggr|\]
\[\texttt{var}_{Mx}^j=\frac{1}{5}\sum_{i=0}^4\Biggl|\frac{\|M_x^{j-i}\|-\|M_x^{j-i-1}\|}{\|M_x^{j-i-1}\|}\Biggr|,\quad\texttt{var}_{Ms}^j=\frac{1}{5}\sum_{i=0}^4\Biggl|\frac{\|M_s^{j-i}\|-\|M_s^{j-i-1}\|}{\|M_s^{j-i-1}\|}\Biggr|\]
These are the average relative variations, in the last five inner iterations, of the quantities $p_\text{inf}$, $d_\text{inf}$, $M_x$ and $M_s$. From the previous Figure, one expects these quantities to decrease during the CG iterations and, since they are related to the IPM convergence, to be better indicators than the simple relative residual of the linear system. For some problems it can be useful to consider also the indicator $\texttt{var}_\mu$ defined in the same way as before but considering the complementarity measure $\mu$.

Figure~\ref{indicators_comparison} shows the proposed indicators compared to the relative residual, at every CG iteration, during the computation of various IPM directions. Notice that these are only some of the behaviours that were observed; the purpose of these images is to show that the indicators can sometimes decrease similarly to the residual, while on occasions they may display an erratic behaviour, which is difficult to capture looking only at the residual. 

\begin{figure}[h]
\caption{Various behaviours of the proposed indicators compared to the relative residual}
\label{indicators_comparison}
\centering
\subfloat[]{\includegraphics[width=.3\textwidth]{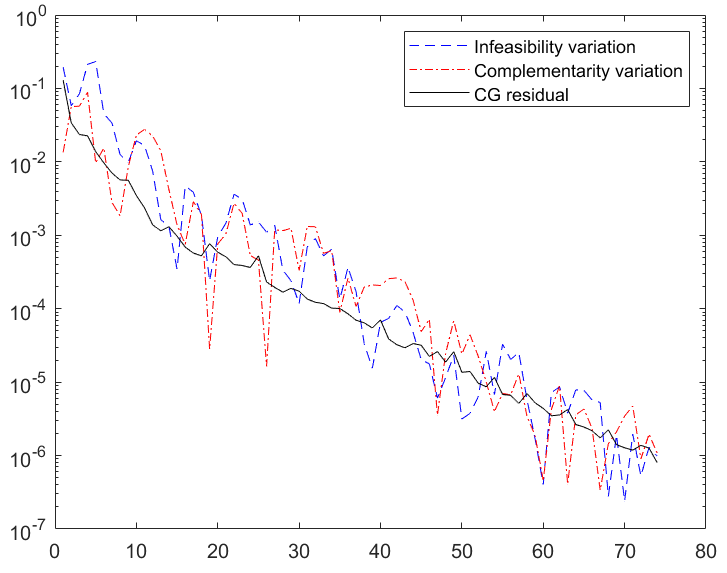}}\quad
\subfloat[]{\includegraphics[width=.3\textwidth]{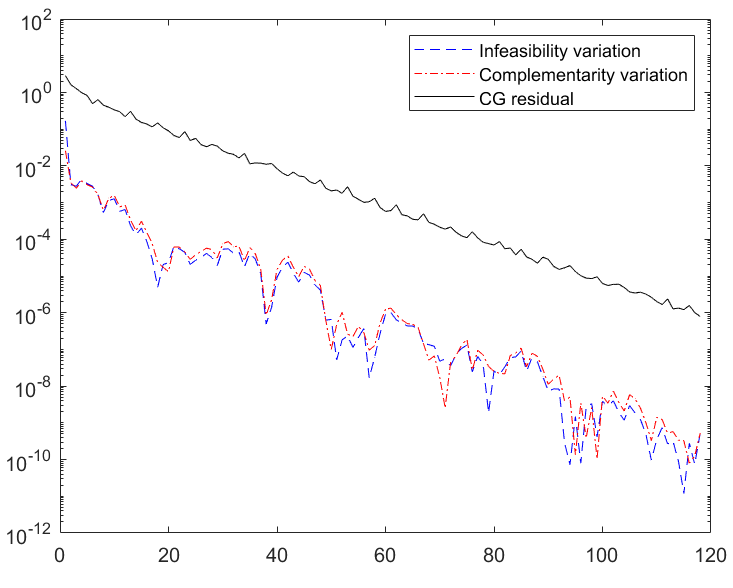}}\\
\subfloat[]{\includegraphics[width=.3\textwidth]{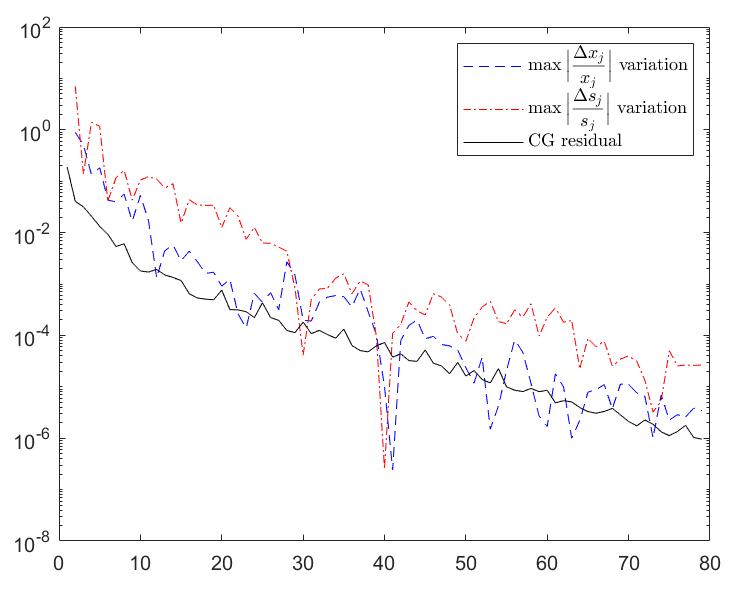}}\quad
\subfloat[]{\includegraphics[width=.3\textwidth]{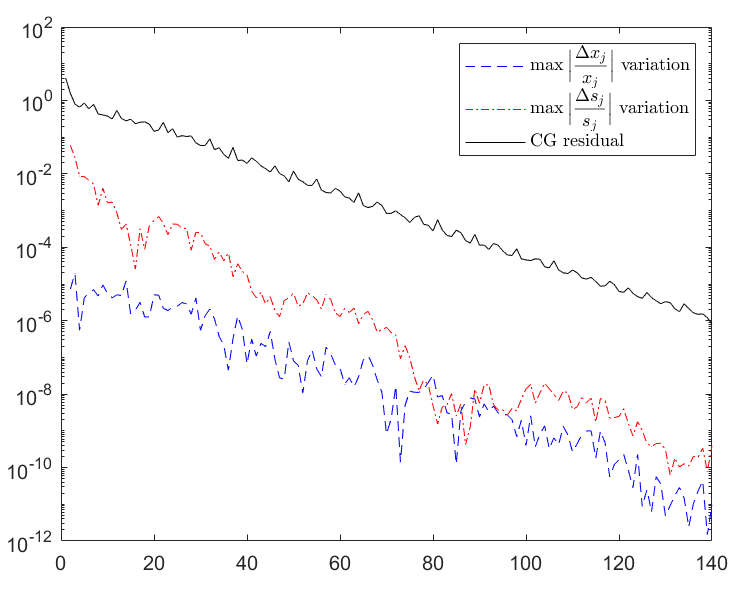}}
\end{figure}

To summarize, the following termination criterion is proposed: the inner iterations are stopped if
\[\Big((\texttt{var}_P^j<\varepsilon)\,\, \wedge\,\, (\texttt{var}_D^j<\varepsilon)\,\, \wedge\,\, (\texttt{var}_{Mx}^j<\varepsilon)\,\, \wedge\,\, (\texttt{var}_{Ms}^j<\varepsilon)\Big)\,\,\vee\,\,\big(\|r\|/\|r_0\|<\tau_\text{inner}\big).\]
The first four conditions check if the new indicators are all smaller than a tolerance $\varepsilon$; however, if it happens that the residual gets sufficiently small before the new indicators do, then the stopping criterion is triggered anyway, as with a standard residual test.

The practical criterion presented here is clearly a different technique than the theoretical one shown in Algorithm~\ref{ipm_alg}; however, it is strongly inspired by the arguments of the previous Section. The indicators considered involve the same quantities used in the criterion \eqref{ipcg_criterion_1}-\eqref{ipcg_criterion_2} and the condition of small relative variations ensures that the inexact direction found is likely to produce a point that gets close to satisfying the theoretical criterion as well. Some safeguards are required in order to keep the behaviour of the practical criterion close to the ideal one: in particular, the practical criterion may be triggered by chance in the very first linear iterations, when the theoretical stopping criterion is not yet satisfied. For this reason, the parameter {\tt itstart} is important since it prevents this phenomenon from happening.

Such a difference between theoretical and practical methods is not unusual in the IPM literature, where often the theoretical properties are proven for the methods, but to achieve the best performance the practical algorithms slightly deviate from the rigorous theoretical settings.


\section{Numerical results}
\label{section_results}
In this section, the test problems are introduced and the results obtained with the standard CG or MINRES and with the novel IPCG or IPMINRES are presented. This section shows overall results in terms of IPM iterations, inner iterations and computational time, and then provides also an insight into the individual IPM iterations to demonstrate where the gains resulting from the new method are the most significant. 

The numerical experiments were performed using MATLAB R2018a and were run on the University of Edinburgh School of Mathematics computing server, which is equipped with four 3.3GHz octa-core Intel Gold 6234 processors and 500GB of RAM; the experiments never used more than 4 cores and 20GB of memory.

The new technique is compared with two options that are usually employed when dealing with a Krylov method inside an IPM: a fixed tolerance on the relative residual of the linear system and a variable tolerance proportional to the complementarity measure $\mu$ (see e.g.\ \cite{stop_inner,gondzio_inexact,inexact_morini}). In particular, the tolerance for the second option is chosen at each IPM iteration $k$ as
\[\tau_\text{inner}^k=\max(\text{{\tt tol}}^\text{max},\frac{\mu^k}{\mu^0}\,\text{{\tt tol}}^0)\]
where $\text{{\tt tol}}^0$ is the initial tolerance, $\text{{\tt tol}}^\text{max}$ is the best allowed tolerance, $\mu^k$ is the value of the complementarity measure at the current iteration and $\mu^0$ is the initial one. In this way, the tolerance decreases at the same rate as $\mu$ until it reaches the value $\text{{\tt tol}}^\text{max}$. In the following, these two options are denoted as {\tt fixtol} and {\tt vartol} respectively.

Despite having multiple options available to choose a variable tolerance, the authors compared the results with this one, since it is very simple and widespread. There may be other tolerance sequences, tailored specifically to the problem considered, that produce better results; however, the new criterion that is introduced does not need to be redesigned for a specific problem and hence a variable tolerance was selected in the same simple way for all problems. It is worth pointing out that other specialized stopping criteria, developed for different problems (e.g. \cite{stop_fem,stop_axelsson,freitag_spence,semiconvergence,notay_cg,silvester_simoncini,stathopoulos}) cannot be easily generalized and used inside an IPM, since the quantities used for these criteria may not even have a meaningful interpretation in this context.

 The values of $\text{{\tt tol}}^0$ and $\text{{\tt tol}}^\text{max}$ were chosen after a quick tuning process in order to obtain the best results with the variable tolerance method; the same holds for the parameters $\varepsilon$ and {\tt itstart} of the new stopping criterion. The specific values are given below for each problem class.

\subsection{Tomographic reconstruction}
The first test problem involves the reconstruction of an image obtained with a dual-energy x-ray tomography \cite{xray}. This is a classical inverse problem in many practical fields, from medicine to industrial applications. The noise in the measurements and the requirement of using as few angles of measurement as possible (e.g.\ to minimize the radiation dose to a patient), make this kind of problem challenging. The goal is to understand a spatial distribution of two different materials, for example the bone and soft tissue; to do so, the domain of interest is discretized and the information about the concentration of the two materials in the points of the discretization is stored in two vectors $x_1, x_2\in\mathbb R^n$. In \cite{xray}, the authors propose a new regularization technique which replaces the standard Joint Total Variation approach and exploits the inner product $x_1^Tx_2$ to enforce the separation of the two materials. 

Stacking together the vectors $x_1$ and $x_2$ into a single vector $x\in\mathbb R^{2n}$, the optimization problem that arises takes the following form
\[\min_{x\ge0}\ \|w-\mathcal Gx\|^2+\rho\|x\|^2+2\eta\,x_1^Tx_2,\] 
where $w$ is the measurement vector and $\mathcal G$ is an operator that incorporates information about the geometry of the problem and the materials used; $\rho$ is the coefficient for the Tikhonov regularization and $\eta$ for the novel regularizer. Written as a standard QP, the problem reads
\[\min_{x\ge0}\ \frac{1}{2}x^T\mathcal Qx-w^T\mathcal Gx,\] 
where
\[\mathcal Q=\begin{bmatrix} c_{11}^2+c_{21}^2\ & c_{11}c_{12}+c_{21}c_{22}\\c_{11}c_{12}+c_{21}c_{22}\ & c_{12}^2+c_{22}^2\end{bmatrix}\otimes R^TR+\begin{bmatrix} \rho\ & \eta\\\eta\ & \rho\end{bmatrix}\otimes I.\]
Here $c_{11}$, $c_{12}$, $c_{21}$ and $c_{22}$ describe the attenuation constants of the two materials for the two x-ray energies used, while $R$ contains information about the geometry of the measurements and can only be accessed via matrix-vector products performed using the Radon transform; $\otimes$ denotes the Kronecker product. 

This optimization problem does not have linear equality constraints; if one applies an interior point method and formulates the normal equations, the linear system that arises has matrix $\mathcal Q+X^{-1}S$. The structure of matrix $R^TR$ allows the use of a block-diagonal preconditioner
\[P=\begin{bmatrix}(c_{11}^2+c_{21}^2)\nu I+\rho I & (c_{11}c_{12}+c_{21}c_{22})\nu I+\eta I\\ (c_{11}c_{12}+c_{21}c_{22})\nu I+\eta I & (c_{12}^2+c_{22}^2)\nu I+\rho I\end{bmatrix}+X^{-1}S,\]
where $\nu$ approximates the main diagonal of the blocks in $R^TR$. Therefore, it is possible to apply the CG with this positive definite preconditioner to find the IPM direction. The application of the matrix of the system is particularly expensive, since it involves the call of the Radon and inverse Radon transforms, to apply $R$ and $R^T$ respectively; thus, a single CG iteration is particularly expensive and the authors expect the IPCG to bring a substantial benefit.
 
An IPM with centrality correctors was applied to this problem: the IPM tolerance was set to $10^{-8}$, the CG tolerance for the {\tt fixtol} approach was $10^{-6}$ and the parameters for the {\tt vartol} approach were $\texttt{tol}^\text{max}=10^{-6}$, $\texttt{tol}^0=10^{-3}$. These parameters were selected because they allow fewer linear iterations, without compromising too much the quality of the inexact direction and the IPM convergence speed. The new IPCG approach was applied with parameters $\varepsilon=0.01$ and $0.001$, $\texttt{itstart}=5$ and $\tau_\text{inner}=10^{-6}$. Since the problems contain noise that is randomly initialized at every run, the results shown are the average over 10 runs, for each discretization level.

Table \ref{results_ue_standard} reports the results using the {\tt fixtol} and {\tt vartol} approaches. The parameter {\tt level} indicates how fine the discretization is; the size of the matrix is equal to $2\cdot\texttt{level}^2$, so that the largest instance has $524,288$ variables.

\begin{table}[h]
\footnotesize
\caption{Results with CG: {\tt fixtol} and {\tt vartol}}
\label{results_ue_standard}
\centering
\begin{tabular}{r|rrr|rrr}
\toprule
&\multicolumn{3}{c}{CG {\tt fixtol}}&\multicolumn{3}{c}{CG {\tt vartol}}\\
\cmidrule(lr){2-4}
\cmidrule(lr){5-7}
{\tt level} & IPM It & Inner It & Time & IPM It & Inner It & Time\\
\midrule
32 & 16.7 & 3,892.1 & 9.02 & 17.1 & 2,001.5 & 4.71\\
64 & 21.0 & 6,436.9 & 27.37 & 21.4 & 3,350.0 & 14.64\\
128 & 22.5 & 9,514.8 & 104.96 & 26.0 & 4,952.6 & 55.48\\
256 & 24.8 & 14,183.5 & 511.87 & 33.0 & 8,059.3 & 295.45\\
512 & 29.5 & 21,897.7 & 3,035.63 & 44.0 & 14,413.0 & 1,954.11\\
\bottomrule
\end{tabular}
\end{table}

Tables \ref{results_ue_ipcg1} and \ref{results_ue_ipcg2} instead show the results obtained with IPCG with $\varepsilon=0.01$ and $\varepsilon=0.001$ respectively; the last columns show the reduction in linear iterations and computational time when compared with the previous approaches.

\begin{table}[h]
\footnotesize
\caption{Results with IPCG ($\varepsilon=0.01$)}
\label{results_ue_ipcg1}
\centering
\begin{tabular}{r|rrr|rr|rr}
\toprule
&\multicolumn{3}{c}{IPCG}&\multicolumn{2}{c}{Reduction {\tt fixtol}} &\multicolumn{2}{c}{Reduction {\tt vartol}}\\
\cmidrule(lr){2-4}
{\tt level} & IPM It & Inner It & Time & Inner It \% & Time \%& Inner It \% & Time \%\\
\midrule
32 & 16.7 & 842.2 & 2.11 & 78.3 & 76.6 & 57.9 & 55.2\\
64 & 20.5 & 1,329.1 & 6.67 & 79.4 & 75.6 & 60.3 & 54.4\\
128 & 22.7 & 1,688.1 & 22.73 & 82.2 & 78.3 & 65.9 & 59.0\\
256 & 26.8 & 2,090.2 & 93.46 & 85.3 & 81.7 & 74.1 & 68.4\\
512 & 34.0 & 2,849.1 & 509.11 & 87.0 & 83.2 & 80.2 & 73.9\\
\bottomrule
\end{tabular}
\end{table}

\begin{table}[h]
\footnotesize
\caption{Results with IPCG ($\varepsilon=0.001$)}
\label{results_ue_ipcg2}
\centering
\begin{tabular}{r|rrr|rr|rr}
\toprule
&\multicolumn{3}{c}{IPCG}&\multicolumn{2}{c}{Reduction {\tt fixtol}} &\multicolumn{2}{c}{Reduction {\tt vartol}}\\
\cmidrule(lr){2-4}
{\tt level} & IPM It & Inner It & Time & Inner It \% & Time \%& Inner It \% & Time \%\\
\midrule
32 & 16.5 & 1,219.6 & 3.07 & 68.7 & 66.0 & 39.0 & 34.8\\
64 & 20.5 & 1,870.8 & 8.77 & 70.9 & 68.0 & 44.1 & 40.1\\
128 & 22.0 & 2,538.2 & 33.82 &  73.3 & 67.8 & 48.8 & 39.0\\
256 & 25.5 & 3,475.7 & 152.67 & 75.5 & 70.2 & 56.9 & 48.3\\
512 & 30.3 & 4,950.0 & 859.22 & 77.4 & 71.7 & 65.7 & 56.0\\
\bottomrule
\end{tabular}
\end{table}

It is worth observing that when using IPM with the new stopping criterion, the number of outer (IPM) iterations is very close to the one obtained with the original IPM using {\tt fixtol}; this confirms that the inexact direction is sufficiently precise so as not to destroy the convergence properties of IPM. In particular, it can be noticed that using a lower tolerance $\varepsilon$ guarantees an IPM iteration count almost identical to the original one;  the {\tt vartol} approach instead produces a substantial increase in the IPM iterations, particularly for larger problems.

Observe also that the IPCG with $\varepsilon=0.01$ produces a similar number of IPM iterations as the {\tt vartol} approach, but uses far fewer inner iterations. This means that the new technique is better at choosing when to stop the linear iterations and does not compromise the overall IPM convergence more than a standard inexact IPM would.

The reduction in terms of linear iterations is very high and reaches values of more than $70$\% for both choices of $\varepsilon$ for the largest instance considered. This translates into a significant computational time reduction, which confirms that the operations added inside the IPCG algorithm are very cheap. Indeed, the time per CG iteration when $\texttt{level}=512$ goes roughly from 140ms in the case of standard CG to 175ms in the case of IPCG; a small increase which is offset by a large reduction in the number of inner iterations.
\begin{figure}[h]
\caption{{\bf (a)} CG iterations per IPM iteration and {\bf (b)} final relative residual for the three approaches considered}
\centering
\subfloat[]{\label{iter_comparison_ue}\includegraphics[width=.48\textwidth]{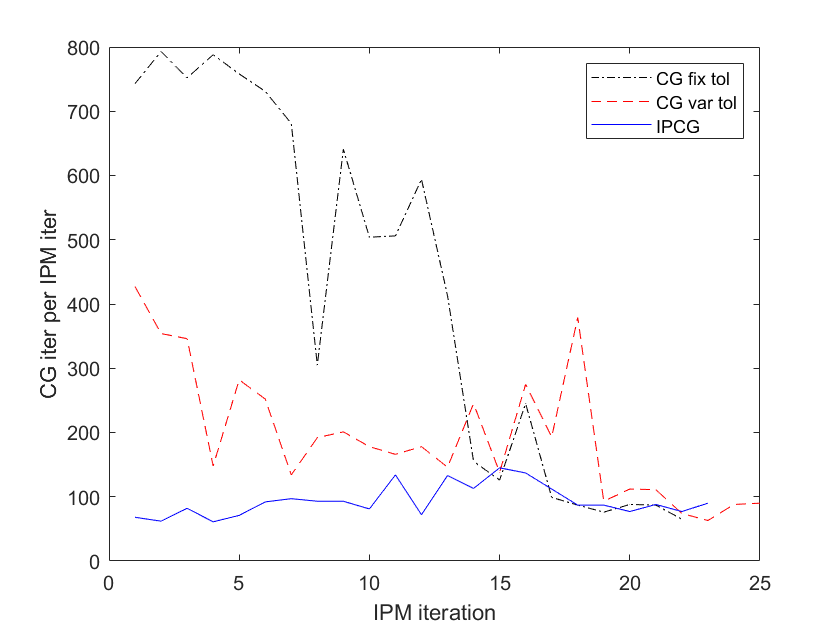}}
\subfloat[]{\label{residual_comparison_ue}\includegraphics[width=.48\textwidth]{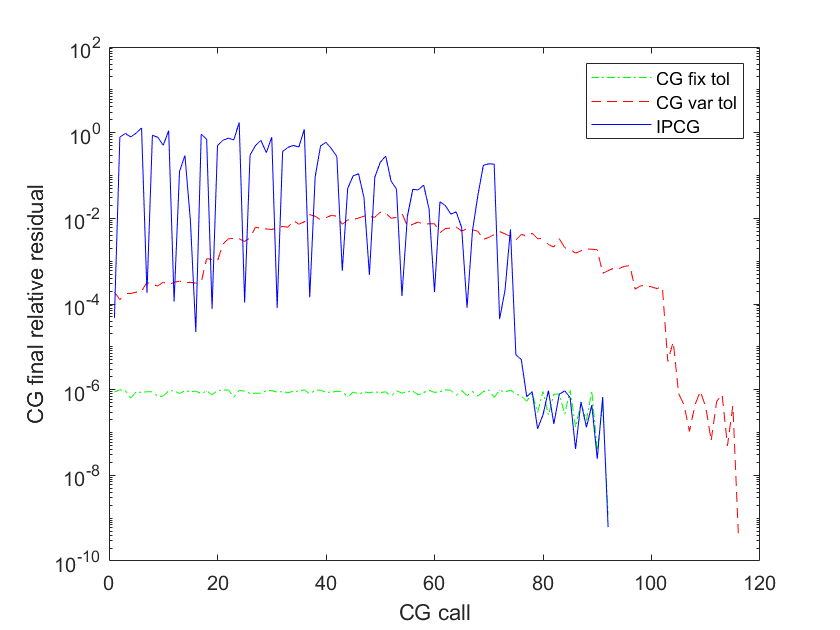}}
\end{figure}

Next, it is useful to understand how the gain of IPCG is distributed during the IPM iterations. To do this, the authors recorded the number of CG iterations at each IPM iteration (summing together the inner iterations for predictor and correctors) in three different situations: using CG with {\tt fixtol}; using CG with {\tt vartol}; using IPCG with $\varepsilon=10^{-3}$. Figure~\ref{iter_comparison_ue} shows the comparison of the iterations for the problem with $\texttt{level}=128$.

Notice that, when using standard CG with fixed tolerance, the number of iterations decreases at the end, since less correctors are computed; when using IPCG, this decrease is not observed, since the smaller number of correctors is balanced by the increased accuracy needed. Indeed, in the late IPM phase, the new stop criterion is not triggered and IPCG stops with the standard reduction test; the reader may observe that the two graphs overlap in the last iterations. However, in the initial phase, a significant advantage of IPCG over the standard CG can be noticed, both for the {\tt fixtol} and {\tt vartol} approaches. It is curious how the IPCG requires almost every time less inner iterations than the CG with {\tt vartol}, but still manages to converge in a smaller number of IPM iterations. This is because the number of inner iterations used for the predictor and for the correctors is distributed differently: the standard CG applies the same tolerance to all the directions during a certain IPM iteration, while IPCG chooses when to stop the inner iterations based on the improvement that the direction can bring to the IPM convergence. In this way, some correctors are computed very roughly, without spoiling the overall IPM convergence speed.

This is clear from the next analysis that was performed: the final relative residual at each CG call (for predictors and correctors) was recorded for all the three approaches; the results are shown in Figure~\ref{residual_comparison_ue}. Notice that the IPCG computes directions both more accurately and less accurately than the {\tt vartol} approach, depending on how much a certain direction is able to improve the quality of the IPM point; the surprising variability of the final residual suggests that there is much to be gained by an approach that does not involve only the residual tolerance, because otherwise it would not be possible to capture this behaviour. This graph highlights also that no stopping criterion based on a residual tolerance (potentially different from the {\tt vartol} approach considered here) could match the performance of the proposed solver, given the variability observed.

The graphs displayed in these two figures undeniably confirm that a high accuracy in the first IPM iterations is not needed at all, and that the best method to decide when a direction is sufficiently precise to perform the next IPM iteration successfully should be based on the IPM indicators and not on the residual of the linear system.

\subsection{Compressed sensing}
The second test problem arises from compressed sensing \cite{mfcs}: a sparse solution to an undetermined linear system $Ax=b$ is sought, where sparsity is enforced by means of a 1-norm regularization. After linearizing the 1-norm by adding extra variables, the optimization problem that arises is the following
\[\min_{z\ge0}\tau e^Tz+\frac{1}{2}\|F^Tz-b\|^2,\]
where $\tau>0$, $z=\begin{bmatrix}u\ ;&v\end{bmatrix}$, $u$ and $v$ being the positive and negative parts of vector $x$, and $F^T=\begin{bmatrix}A\ & -A\end{bmatrix}$. Rewriting it as a standard quadratic program and formulating the IPM normal equations, the matrix of the linear system to be solved is
\[H=\begin{bmatrix}1\ & -1\\-1\ & 1\end{bmatrix}\otimes A^TA+\Theta^{-1}.\]
Due to the structure of matrix $A$, which satisfies the restricted isometry property (see \cite{mfcs} for all the details), matrix $H$ can be efficiently preconditioned by the block diagonal matrix
\[P=\begin{bmatrix}1\ &-1\\-1\ &1\end{bmatrix}\otimes\eta I+\Theta^{-1}\]
for an appropriate constant $\eta$. The difference with respect to the first test problem is that now the IPM direction is computed using a very low accuracy for the CG: the residual tolerance is $10^{-1}$ or $10^{-2}$, depending on the problem, throughout all the IPM iterations.  Due to this very rough tolerance, the {\tt vartol} approach was not able to bring any substantial improvement. For this class of problems only the {\tt fixtol} approach was used.

The test problems are taken from the Sparco collection \cite{sparco}; of the 18 problems considered in \cite{mfcs}, 5 did not show any improvement when using IPCG instead of CG (in part because they were easy enough and the CG was already performing a low number of iterations). In Table~\ref{cs_results} the results for the remaining 13 that did show an improvement are presented. The IPM tolerance varies between $10^{-6}$ and $10^{-10}$ according to the problem being solved and no corrector direction is used. The default values for IPCG are $\varepsilon=0.01$ and $\texttt{itstart}=5$, but some problems required different parameters, which are indicated in the Table. 

\begin{table}[t]
\footnotesize
\caption{Results for Compressed sensing IPM.}
\label{cs_results}
\centering
\begin{tabular}{rr|rrr|rrr|rr}
\toprule
&&\multicolumn{3}{c}{CG {\tt fixtol}}&\multicolumn{3}{c}{IPCG}&\multicolumn{1}{c}{Inner It}&\multicolumn{1}{c}{Time}\\
\cmidrule(lr){3-5}
\cmidrule(lr){6-8}
ID & Size & IPM & Inner It &Time& IPM & Inner It &Time&red \%&red \% \\
\midrule
6 & 4,096 & 22 & 2,128 & 40.22 & 23 & 193 & 4.11 & 90.9 & 89.8\\
9 & 256 & 11 & 382 & 0.23 & 11 & 147 & 0.13 & 61.5 & 43.5\\
10 & 2,048 & 12 & 2,210 & 0.57 & 16 & 874 & 0.34 & 60.5 & 40.4\\
11$\,\dagger$* & 2,048 & 19 & 663 & 1.41 & 21 & 536 & 1.14 & 19.2 & 19.1\\
401 & 114,688 & 14 & 160 & 15.72 & 12 & 55 & 7.03 & 65.6 & 55.3\\
402$\,\dagger$ & 172,032 & 14 & 238 & 28.14 & 12 & 59 & 9.76 & 75.2 & 65.3\\
403 & 393,216 & 19 & 2,282 & 201.44 & 20 & 205 & 36.94 & 91.0 & 86.6\\
601$\,\dagger$ & 8,192 & 20 & 2,146 & 104.20 & 21 & 652 & 28.41 & 69.6 & 72.7\\
602 & 8,192 & 22 & 2,280 & 124.39 & 20 & 453 & 19.12 & 80.1 & 84.6\\
603 & 8,192 & 16 & 1,085 & 16.53 & 13 & 86 & 2.13 & 92.1 & 87.1\\
701$\,\dagger$ & 131,072 & 12 & 1,028 & 38.64 & 12 & 236 & 13.73 & 77.0 & 64.5\\
702 & 32,768 & 8 & 926 & 15.00 & 8 & 181 & 4.42 & 80.5 & 70.5\\
903$\,\dagger$ & 2,048 & 13 & 1,794 & 2.52 & 16 & 687 & 0.93 & 61.7 & 63.1\\
\midrule
\multicolumn{9}{c}{\footnotesize$\dagger$: $\varepsilon=0.001$ instead of $0.01$,\ \ *: $\texttt{itstart}=20$ instead of 5}\\
\bottomrule
\end{tabular}
\end{table}

All these problems display an impressive reduction in the number of CG iterations and CPU time, even if the original CG tolerance is very rough. The added cost of IPCG varies throughout the problems, but on average is roughly $35-40\%$ of the original iteration cost. Sometimes a reduction in IPM iterations is also observed; this may be because the inexact method proposed is finding by chance a direction that is better than the exact one.

\subsection{PDE constrained optimization}
As a last test example, the authors considered PDE constrained optimization problems (see e.g.\ \cite{pearson_pde}) and used the augmented system approach, in order to test Algorithm \ref{ipminres_alg}. In this section, $\hat v$ and $v$ will denote respectively the continuous and discretized version of a variable $v$. The kind of problems considered involve PDE as constraints and they take the standard form
\begin{align}
\min_{y,u}\quad \frac{1}{2}\|\hat y&-\hat y_0\|^2_{L^2}+\frac{\beta}{2}\|\hat u\|^2_{L^2}\notag\\
\text{s.t.}\ -\nabla^2&\hat y=\hat u+\hat f,\quad \hat y\in\Omega\notag\\
 &\hat y=\hat g,\quad \hat y\in\partial\Omega\notag\\
&\hat u_a\le \hat u\le \hat u_b\notag
\end{align}
where $\Omega$ is the domain of evolution of the problem, $\hat y, \hat u$ are the state and control variables, $\hat y_0$ is the desired state function, $\hat f, \hat g, \hat u_a, \hat u_b$ are given functions and $\beta>0$ is the regularization parameter. The objective of this formulation is to keep the state $\hat y$ close to the fixed desired state $\hat y_0$ and minimize the control $\hat u$, while satisfying the PDE and bound constraints. Problems of this kind arise, for example, in optimal control theory: practical applications include optimal design of semiconductors, shape optimization, optimal gas cooling and many others (the interested reader can find more details in \cite{pde_applications}).

A standard IPM is applied to this problem, using the discretize-than-optimize approach, as described in \cite{pearson_pde}, to obtain the discretized quantities $y, u,  y_0, u_a, u_b$; then, introduce the variables $z_a$ and $z_b$ defined as $(z_a)_j=\mu/(u-u_a)_j$ and $(z_b)_j=\mu/(u_b-u)_j$. After using a standard Q1 finite elements discretization, the augmented system reads
\[\begin{bmatrix} M\ & 0\ & K\\0\ & \beta M+\Theta\ & -J\\ K\ & -J\ & 0\end{bmatrix}\begin{bmatrix}\Delta y\\\Delta u\\\Delta\lambda\end{bmatrix}=\begin{bmatrix}r_y\\r_u\\r_\lambda\end{bmatrix},\]
where $M\in\mathbb R^{n\times n}$ is the finite elements mass matrix, $J\in\mathbb R^{n\times n}$ is the same matrix but with boundary conditions applied, $K\in\mathbb R^{n\times n}$ is the stiffness matrix, $\Theta=Z_a(U-U_b)^{-1}+Z_b(U_a-U)^{-1}$, $\lambda\in\mathbb R^n$ is the vector of Lagrange multipliers. The dimension of the matrices $n$ is determined by the discretization parameter $n_c$, as $n=(2^{n_c}+1)^2$; the whole augmented system has size $3n$. 

This linear system can be solved using MINRES, provided that the preconditioner is positive definite; exploiting the ideas in \cite{pearson_pde} and \cite{pde_prec}, the following preconditioner is employed
\[P=\begin{bmatrix}\tilde M\ & 0\ & 0\\0\ & \beta\tilde M+\Theta\ & 0\\0\ & 0\ & \tilde S\end{bmatrix},\]
where $\tilde M$ contains only the diagonal of $M$ and $\tilde S$ is an approximation of the Schur complement
\[\tilde S=\bigg(K+\frac{1}{\sqrt{\beta}}J\bigg)M^{-1}\bigg(K+\frac{1}{\sqrt{\beta}}J\bigg).\]
The Schur complement preconditioner is constant throughout the IPM iterations and to apply it it suffices to compute the Cholesky factorization of $(K+J/\sqrt\beta)$ once at the beginning of the algorithm. The finite element matrices were computed using the IFISS package \cite{ifiss,ifiss_paper_2}.

\begin{table}[h]
\footnotesize
\caption{Results with {\tt fixtol} and {\tt vartol} approaches}
\label{results_pde_1}
\centering
\begin{tabular}{cr|rrr|rrr}
\toprule
&&\multicolumn{3}{c}{MINRES {\tt fixtol}}&\multicolumn{3}{c}{MINRES {\tt vartol}}\\
\cmidrule(lr){3-5}
\cmidrule(lr){6-8}
$\beta$ & $n_c$ & IPM & Inner It & Time & IPM & Inner It & Time \\
\midrule
\multirow{5}{*}{$10^{-4}$} & 5 & 10 & 747 & 0.32 & 10 & 455 & 0.19\\
& 6  & 11 & 812 & 2.08 & 11 & 515 & 1.34\\
& 7 & 13 & 919 & 29.76 & 13 & 621 & 19.93\\
& 8 & 14 & 930 & 327.33 & 14 & 655 & 236.58\\
& 9 & 14 & 839 & 5,094.79 & 14 & 672 & 3,722.21\\
\midrule
\multirow{5}{*}{$10^{-5}$} & 5 & 11 & 1,424 & 0.48 & 11 & 782 & 0.27\\
& 6 & 13 & 1,711 & 4.27 & 13 & 978 & 2.48\\
& 7 & 14 & 1,861 & 60.08 & 14 & 1,036 & 32.51\\
& 8 & 16 & 2,037 & 706.28 & 16 & 1,231 & 437.32\\
& 9 & 16 & 1,950 & 11,783.63 & 16 & 1,163 & 6,996.73\\
\midrule
\multirow{5}{*}{$10^{-6}$} & 5 & 14 & 3,511 & 1.09 & 13 & 1,798 & 0.57\\
& 6 & 15 & 3,902 & 9.66 & 15 & 2,217 & 5.55\\
& 7 & 16 & 4,216 & 125.13 & 16 & 2,316 & 72.18\\
& 8 & 17 & 4,450 & 1,530.93 & 17 & 2,346 & 814.10\\
& 9 & 20 & 4,959 & 29,979.89 & 19 & 2,679 & 16,260.61\\
\bottomrule
\end{tabular}
\end{table}

An IPM with centrality correctors was applied to this problem. The parameters used are: IPM tolerance $10^{-8}$; for the {\tt fixtol} approach, MINRES tolerance $10^{-8}$; for the {\tt vartol} approach, $\texttt{tol}^\text{max}=10^{-8}$, $\texttt{tol}^0=10^{-2}$; for the IPCG, $\varepsilon=10^{-3}$ and $\texttt{itstart}=15$. Values of $n_c$ from $5$ to $9$ were considered, so that the largest problem had dimension $789,507$; for $\beta$, the values used were $10^{-4}$, $10^{-5}$ and $10^{-6}$.

Table~\ref{results_pde_1} shows the results using the {\tt fixtol} and {\tt vartol} approaches. Table~\ref{results_pde_2} shows the results using the IPMINRES method and the last columns report the reductions in the number of inner iterations and computational time compared with the previous approaches. A negative reduction means that the new approach produces a larger number of iterations or larger computational time: this happens in some small problems either because the inexact direction produces a large number of IPM iterations, or because the reduction in inner iterations is not enough to balance the more expensive CG iteration. This phenomenon however disappears for larger problems.

\begin{table}[h]
\footnotesize
\caption{Results using IPMINRES}
\label{results_pde_2}
\centering
\begin{tabular}{cr|rrr|rr|rr}
\toprule
&&\multicolumn{3}{c}{IPMINRES}&\multicolumn{2}{c}{Reduction {\tt fixtol}:}&\multicolumn{2}{c}{Reduction {\tt vartol}:}\\
\cmidrule(lr){3-5}
$\beta$ & $n_c$ & IPM & Inner It & Time & Inner It \% & Time \% & Inner It \% & Time \%\\
\midrule
\multirow{5}{*}{$10^{-4}$} & 5 & 13 & 535 & 0.32 & 28.4 & 0.0 & -17.6 & -68.4\\
& 6  & 16 & 653 & 1.85 & 19.6 & 11.2 & -26.8 & -38.9\\
& 7 & 15 & 594 & 19.07 & 35.4 & 35.9 & 4.3 & 4.3\\
& 8 & 16 & 640 & 228.36 & 31.2 & 30.2 & 2.3 & 3.5\\
& 9 & 16 & 630 & 3,534.22 & 24.9 & 20.9 & 6.3 & 5.1\\
\midrule
\multirow{5}{*}{$10^{-5}$} & 5 & 12 & 734 & 0.31 & 48.5 & 35.4 & 6.1 & -14.8\\
& 6 & 14 & 821 & 2.30 & 52.0 & 46.1 & 16.1 & 7.3\\
& 7 & 16 & 862 & 27.36 & 53.7 & 54.5 & 16.8 & 15.8\\
& 8 & 16 & 852 & 299.81 & 58.2 & 57.6 & 30.8 & 31.4\\
& 9 & 18 & 834 & 5,072.63 & 50.8 & 49.0 & 28.3 & 27.5\\
\midrule
\multirow{5}{*}{$10^{-6}$} & 5 & 24 & 2,041 & 0.62 & 41.9 & 24.8 & -13.5 & -8.8\\
& 6 & 22 & 1,861 & 5.19 & 52.3 & 46.3 & 16.1 & 6.5\\
& 7 & 18 & 1,533 & 48.04 & 63.6 & 61.6 & 33.8 & 33.4\\
& 8 & 19 & 1,583 & 550.35 & 64.4 & 64.1 & 32.5 & 32.4\\
& 9 & 18 & 1,318 & 8,058.47 & 73.4 & 73.1 & 50.8 & 50.4\\
\bottomrule
\end{tabular}
\end{table}

The reader can observe that when considering larger problems and smaller values of $\beta$, there is a significant reduction in inner iterations and computational time, while the IPM iteration count is almost constant in all three approaches. The improvement that the new method brings is more significant when the linear system becomes more ill conditioned (larger size and smaller $\beta$); this is not surprising, since it is known that the residual of the linear system can be a misleading indicator for ill conditioned problems. The proposed new approach does not suffer from this issue, as these results suggest, because it is related to the IPM properties rather than to the algebraic properties of the matrix, thus making it a more suitable termination criterion for ill conditioned matrices.

These last results show that also the IPMINRES method works as expected and can potentially bring a significant improvement. Moreover, they also show that this new technique of early termination can be applied to different classes of problems, with similar results.

\section{Conclusion}

This paper has shown that it is possible to stop the inner Krylov iterations during an interior point method earlier than it was previously thought, provided that the stopping criterion used is based on the IPM convergence indicators and not only on the reduction of the residual of the linear system.  The authors have given a rationale to explain the expected effect of the termination criterion and have proposed two practical algorithms for the normal equations and augmented system approaches.  They exploit new indicators, related to the convergence of the outer iterations,  and are only marginally more computationally expensive then the original algorithms. A proof of polynomial complexity of such inexact IPM is still elusive and is the subject of further research, as well as a characterization of the constant $M$ involved in the criterion, potentially depending on the linear solver chosen. This could provide a theoretical result on the minimum threshold of accuracy needed for the convergence of IPMs.

The paper provided computational evidence for a wide range of problems, from image processing, compressed sensing and PDE-constrained applications; they all display a significant reduction in the number of inner Krylov iterations and computational time. In particular, the largest gain appears in the early IPM phase, where it is already known that a lower accuracy of Newton directions is sufficient; however, the authors have also shown that it is extremely difficult to mimic the behaviour of the proposed stopping criterion using only a residual test, since the residual of the optimal stopping point may vary drastically during the IPM iterations.  Indeed, the new technique outperforms also the termination criterion that uses a variable residual tolerance. Moreover, the new IPCG seems to keep the IPM iteration count closer to the original one than with a variable tolerance. This fact strongly supports the initial claim that a good stopping criterion for CG or MINRES should be based on the IPM convergence indicators.

The analysis of the numerical results suggests that for ill conditioned problems the performance gain of the new stopping criterion is larger. However, there are some problems that are so badly conditioned and/or require so much precision in the IPM direction that the new stopping criterion is not able to perform well; more research is needed to find a suitable more advanced termination strategy for these challenging problems.

The authors strongly believe that many other practical optimization algorithms 
in which a Krylov subspace method is used to solve the linear equation 
systems are likely to benefit from a specialized stopping criterion 
developed with an understanding of the specific needs of the method.

\subsection*{Acknowledgements}
The authors are grateful to John W. Pearson, for providing codes and fruitful discussions about PDE problems, and to the two anonymous referees, whose comments helped to make the paper stronger and more precise.

\vspace{20pt}
\bibliographystyle{siamplain}
\bibliography{biblio}

\begin{thebibliography}{10}

\bibitem{ifiss}
{\em I{FISS}}.
\newblock https://personalpages.manchester.ac.uk/staff/david.silvester/ifiss/.

\bibitem{minres_web}
{\em M{INRES}}.
\newblock https://web.stanford.edu/group/SOL/software/minres/.

\bibitem{quadreg}
{\sc A.~Altman and J.~Gondzio}, {\em Regularized symmetric indefinite systems
  in interior point methods for linear and quadratic optimization},
  Optimization Methods and Software, 11-12 (1999), pp.~275--302.

\bibitem{regmi2}
{\sc I.~K. Argyros and S.~Regmi}, {\em Undergraduate {R}esearch at {C}ameron
  {U}niversity on {I}terative {P}rocedures in {B}anach and {O}ther {S}paces},
  Nova Science Publisher, New York, USA, 2019.

\bibitem{stop_fem}
{\sc M.~Arioli}, {\em A stopping criterion for the conjugate gradient algorithm
  in a finite element method framework}, Numerische Mathematik, 97 (2004),
  pp.~1--24.

\bibitem{stop_axelsson}
{\sc O.~Axelsson and I.~Kaporin}, {\em Error norm estimation and stopping
  criteria in preconditioned conjugate gradient iterations}, Numerical Linear
  Algebra with Applications, 8 (2001), pp.~265--286.

\bibitem{baryamureeba}
{\sc V.~Baryamureeba and T.~Steihaug}, {\em On the {C}onvergence of an
  {I}nexact {P}rimal-{D}ual {I}nterior {P}oint {M}ethod for {L}inear
  {P}rogramming}, Large-Scale Scientific Computing, Lecture Notes in Comput.
  Sci., 3743 (2006), pp.~629--637.

\bibitem{inexactipm}
{\sc S.~Bellavia}, {\em Inexact {I}nterior-{P}oint {M}ethod}, Journal of
  Optimization Theory and Applications, 96 (1998), pp.~109--121.

\bibitem{ip_pmm}
{\sc L.~Bergamaschi, J.~Gondzio, A.~Martínez, J.~W. Pearson, and
  S.~Pougkakiotis}, {\em A new preconditioning approach for an interior
  point-proximal method of multipliers for linear and convex quadratic
  programming}, Numerical Linear Algebra with Applications,  (2021), p.~e2361.

\bibitem{berga_indefinite}
{\sc L.~Bergamaschi, J.~Gondzio, and G.~Zilli}, {\em Preconditioning
  {I}ndefinite {S}ystems in {I}nterior {P}oint {M}ethods for {O}ptimization},
  Computational Optimization and Applications, 28 (2004), pp.~149--171.

\bibitem{iter:BCO-COAP}
{\sc S.~Bocanegra, F.~Campos, and A.~Oliveira}, {\em Using a {H}ybrid
  {P}reconditioner for {S}olving {L}arge-{S}cale {L}inear {S}ystems arising
  from {I}nterior {P}oint {M}ethods}, Computational Optimization and
  Applications, 36 (2007), pp.~149--164.

\bibitem{stop_inner}
{\sc S.~Cafieri, M.~D'Apuzzo, V.~De~Simone, and D.~di~Serafino}, {\em Stopping
  criteria for inner iterations in inexact potential reduction methods: a
  computational study}, Computational Optimization and Applications, 36 (2007),
  pp.~165--193.

\bibitem{cafieri2}
{\sc S.~Cafieri, M.~D'Apuzzo, V.~De~Simone, D.~di~Serafino, and G.~Toraldo},
  {\em Convergence {A}nalysis of an {I}nexact {P}otential {R}eduction {M}ethod
  for {C}onvex {Q}uadratic {P}rogramming}, Journal of Optimization Theory and
  Applications, 135 (2007), pp.~355--366.

\bibitem{colombo_gondzio}
{\sc M.~Colombo and J.~Gondzio}, {\em Further {D}evelopment of {M}ultiple
  {C}entrality {C}orrectors for {I}nterior {P}oint {M}ethods}, Computational
  Optimization and Applications, 41 (2008), pp.~277--305.

\bibitem{cui_morikuni}
{\sc Y.~Cui, K.~Morikuni, T.~Tsuchiya, and K.~Hayami}, {\em Implementation of
  interior-point methods for {L}{P} based on {K}rylov subspace iterative
  solvers with inner-iteration preconditioning}, Computational Optimization and
  Applications, 74 (2019), pp.~143--176.

\bibitem{ddd_mutual}
{\sc M.~D'Apuzzo, V.~De~Simone, and D.~di~Serafino}, {\em On mutual impact of
  numerical linear algebra and large-scale optimization with focus on interior
  point methods}, Computational Optimization and Applications, 45 (2010),
  pp.~283--310.

\bibitem{diserafino_orban}
{\sc D.~di~Serafino and D.~Orban}, {\em Constraint-{P}reconditioned {K}rylov
  {S}olvers for {R}egularized {S}addle-{P}oint {S}ystems}, SIAM Journal on
  Scientific Computing, 43 (2021), pp.~A1001--A1026.

\bibitem{ifiss_paper_2}
{\sc H.~C. Elman, A.~Ramage, and D.~J. Silvester}, {\em I{FISS}: {A}
  computational laboratory for investigating incompressible flow problems},
  SIAM Review, 56 (2014), pp.~261--273.

\bibitem{mfcs}
{\sc K.~Fountoulakis, J.~Gondzio, and P.~Zhlobich}, {\em Matrix-free interior
  point method for compressed sensing problems}, Mathematical Programming
  Computation, 6 (2014), pp.~1--31.

\bibitem{freitag_spence}
{\sc M.~A. Freitag and A.~Spence}, {\em Convergence of inexact inverse
  iteration with application to preconditioned iterative solves}, BIT Numerical
  Mathematics, 47 (2007), pp.~27--44.

\bibitem{freund_jarre_mizuno}
{\sc R.~Freund, F.~Jarre, and S.~Mizuno}, {\em Convergence of a class of
  inexact interior point algorithms for linear programs}, Mathematics of
  Operations Research, 24 (1999).

\bibitem{regulipm}
{\sc M.~P. Friedlander and D.~Orban}, {\em A primal-dual regularized
  interior-point method for convex quadratic programs}, Mathematical
  Programming Computation, 4 (2012), pp.~71--107.

\bibitem{semiconvergence}
{\sc S.~Gazzola and M.~Sabaté~Landman}, {\em Krylov methods for inverse
  problems: {S}urveying classical, and introducing new, algorithmic
  approaches}, GAMM-Mitteilungen, 43 (2020).

\bibitem{mfipm}
{\sc J.~Gondzio}, {\em Matrix-{F}ree {I}nterior {P}oint {M}ethod},
  Computational Optimization and Applications, 51 (2012), pp.~457--480.

\bibitem{gondzio_inexact}
{\sc J.~Gondzio}, {\em Convergence {A}nalysis of an {I}nexact {F}easile
  {I}nterior {P}oint {M}ethod for {C}onvex {Q}uadratic {P}rogramming}, SIAM
  Journal on Optimization, 23 (2013), pp.~1510--1527.

\bibitem{xray}
{\sc J.~Gondzio, M.~Lassas, S.-M. Latva-Aijo, S.~Siltanen, and F.~Zanetti},
  {\em Material-separating regularizer for multi-energy x-ray tomography},
  Inverse Problems, 38 (2022).

\bibitem{pde_applications}
{\sc M.~Hinze, R.~Pinnau, M.~Ulbrich, and S.~Ulbrich}, {\em Optimization with
  {PDE} {C}onstraints}, in Mathematical {M}odelling: {T}heory and
  {A}pplications, Springer, 2009.

\bibitem{kelley}
{\sc C.~T. Kelly}, {\em Iterative {M}ethods for {L}inear and {N}onlinear
  {E}quations}, SIAM, Philadelphia, USA, 1995.

\bibitem{korzak}
{\sc J.~Korzak}, {\em Convergence {A}nalysis of {I}nexact
  {I}nfeasible-{I}nterior-{P}oint {A}lgorithms for {S}olving {L}inear
  {P}rogramming {P}roblems}, SIAM Journal on Optimization, 11 (2000).

\bibitem{lu_monteiro_oneal}
{\sc Z.~Lu, R.~Monteiro, and J.~W. O'Neal}, {\em An {I}terative
  {S}olver-{B}ased {I}nfeasible {P}rimal-{D}ual {P}ath-{F}ollowing {A}lgorithm
  for {C}onvex {Q}uadratic {P}rogramming}, SIAM Journal on Optimization, 17
  (2006), pp.~287--310.

\bibitem{mizuno_jarre}
{\sc S.~Mizuno and F.~Jarre}, {\em Global and polynomial-time convergence of an
  infeasible-interior-point algorithm using inexact computation}, Mathematical
  Programming, 84 (1999).

\bibitem{inexact_morini}
{\sc B.~Morini and V.~Simoncini}, {\em Stability and {A}ccuracy of {I}nexact
  {I}nterior {P}oint {M}ethods for {C}onvex {Q}uadratic {P}rogramming}, Journal
  of Optimization Theory and Applications, 175 (2017), pp.~450--477.

\bibitem{unreduced}
{\sc B.~Morini, V.~Simoncini, and M.~Tani}, {\em A comparison of reduced and
  unreduced {K}{K}{T} systems arising from interior point method},
  Computational Optimization and Applications, 68 (2017), pp.~1--27.

\bibitem{notay_cg}
{\sc Y.~Notay}, {\em Combination of {J}acobi–{D}avidson and conjugate
  gradients for the partial symmetric eigenproblem}, Numerical Linear Algebra
  with Applications, 9 (2002), pp.~21--44.

\bibitem{iter:OS-pcg1}
{\sc A.~R.~L. Oliveira and D.~C. Sorensen}, {\em A {N}ew {C}lass of
  {P}reconditioners for {L}arge-{S}cale {L}inear {S}ystems from {I}nterior
  {P}oint {M}ethods for {L}inear {P}rogramming}, Linear Algebra and its
  Applications, 394 (2005), pp.~1--24.

\bibitem{minres}
{\sc C.~C. Paige and M.~A. Saunders}, {\em Solution of sparse indefinite
  systems of linear equations}, SIAM Journal of Numerical Analysis, 12 (1975),
  pp.~617--629.

\bibitem{pearson_pde}
{\sc J.~W. Pearson and J.~Gondzio}, {\em Fast interior point solution of
  quadratic programming problems arising from {P}{D}{E}-constrained
  optimization}, Numerische Mathematik, 137 (2017), pp.~959--999.

\bibitem{pde_prec}
{\sc J.~W. Pearson and A.~J. Wathen}, {\em A new approximation of the {S}chur
  complement in preconditioners for {PDE}-constrained optimization}, Numerical
  Linear Algebra with Applications, 19 (2011), pp.~816--829.

\bibitem{nondiagreg}
{\sc S.~Pougkakiotis and J.~Gondzio}, {\em Dynamic {N}on-diagonal
  {R}egularization in {I}nterior {P}oint {M}ethods for {L}inear and {C}onvex
  {Q}uadratic {P}rogramming}, Journal of Optimization Theory and Applications,
  181 (2019), pp.~905--945.

\bibitem{regmi}
{\sc S.~Regmi}, {\em Optimized {I}terative {M}ethods with {A}pplications in
  {D}iverse {D}isciplines}, Nova Science Publisher, New York, USA, 2021.

\bibitem{silvester_simoncini}
{\sc D.~Silvester and V.~Simoncini}, {\em An {O}ptimal {I}terative {S}olver for
  {S}ymmetric {I}ndefinite {S}ystems {S}temming from {M}ixed {A}pproximation},
  ACM Transactions on Mathematical Software, 37 (2011), pp.~1--22.

\bibitem{stathopoulos}
{\sc A.~Stathopoulos}, {\em Neraly optimal preconditioned methods for
  {H}ermitian eigenproblems under limited memory. {P}art i: seeking one
  eigenvalue}, SIAM J. Sci. Comput., 29 (2007), pp.~481--514.

\bibitem{sparco}
{\sc E.~Van Den~Berg, M.~P. Friedlander, G.~Hennenfent, F.~J. Herrman, R.~Saab,
  and O.~Yilmaz}, {\em Sparco: a testing framework for sparse reconstruction},
  ACM Trans. Math. Softw., 35 (2009), pp.~1--16.

\bibitem{iter:VOC}
{\sc M.~I. Velazco, A.~R.~L. Oliveira, and F.~F. Campos}, {\em A note on hybrid
  preconditioners for large-scale normal equations arising from interior-point
  methods}, Optimization Methods and Software, 25 (2010), pp.~321--332.

\bibitem{wright_qp}
{\sc S.~J. Wright}, {\em A path-following interior-point algorithm for linear
  and quadratic problems}, Annals of Operations Research, 62 (1996),
  pp.~103--130.

\bibitem{wright}
{\sc S.~J. Wright}, {\em Primal-{D}ual {I}nterior-{P}oint {M}ethods}, SIAM,
  1997.

\bibitem{zhou_toh}
{\sc G.~Zhou and K.-C. Toh}, {\em Polynomiality of an inexact infeasible
  interior point algorithm for semidefinite programming}, Mathematical
  Programming, 99 (2004).

\end{thebibliography}

\end{document}